\documentclass[11pt]{article}
\usepackage{latexsym,amsfonts,amssymb,amsmath,amsthm}

\usepackage{color}

\begin{document}

\newtheorem{tm}{Theorem}[section]
\newtheorem{rk}{Remark}[section]
\newtheorem{prop}[tm]{Proposition}
\newtheorem{defin}[tm]{Definition}
\newtheorem{coro}[tm]{Corollary}

\newtheorem{lem}[tm]{Lemma}
\newtheorem{assumption}[tm]{Assumption}

\newtheorem{nota}[tm]{Notation}
\numberwithin{equation}{section}

\newcommand{\stk}[2]{\stackrel{#1}{#2}}
\newcommand{\dwn}[1]{{\scriptstyle #1}\downarrow}
\newcommand{\upa}[1]{{\scriptstyle #1}\uparrow}
\newcommand{\nea}[1]{{\scriptstyle #1}\nearrow}
\newcommand{\sea}[1]{\searrow {\scriptstyle #1}}
\newcommand{\csti}[3]{(#1+1) (#2)^{1/ (#1+1)} (#1)^{- #1
 / (#1+1)} (#3)^{ #1 / (#1 +1)}}
\newcommand{\RR}[1]{\mathbb{#1}}

\newcommand{ \bl}{\color{blue}}
\newcommand {\rd}{\color{red}}
\newcommand{ \bk}{\color{black}}
\newcommand{ \gr}{\color{OliveGreen}}
\newcommand{ \mg}{\color{RedViolet}}

\newcommand{\ep}{\varepsilon}
\newcommand{\rr}{{\mathbb R}}
\newcommand{\alert}[1]{\fbox{#1}}

\newcommand{\eqd}{\sim}
\def\R{{\mathbb R}}
\def\N{{\mathbb N}}
\def\Q{{\mathbb Q}}
\def\C{{\mathbb C}}
\def\l{{\langle}}
\def\r{\rangle}
\def\t{\tau}
\def\k{\kappa}
\def\a{\alpha}
\def\la{\lambda}
\def\De{\Delta}
\def\de{\delta}
\def\ga{\gamma}
\def\Ga{\Gamma}
\def\ep{\varepsilon}
\def\eps{\varepsilon}
\def\si{\sigma}
\def\Re {{\rm Re}\,}
\def\Im {{\rm Im}\,}
\def\E{{\mathbb E}}
\def\P{{\mathbb P}}
\def\Z{{\mathbb Z}}
\def\D{{\mathbb D}}
\newcommand{\ceil}[1]{\lceil{#1}\rceil}

\title{Parabolic-elliptic chemotaxis model with space-time dependent logistic sources on $\mathbb{R}^N$. I. Persistence and asymptotic spreading}

\author{
Rachidi B. Salako and Wenxian Shen  \\
Department of Mathematics and Statistics\\
Auburn University\\
Auburn University, AL 36849\\
U.S.A. }

\date{}
\maketitle
\begin{abstract}
\noindent
The current series of three papers is concerned with the asymptotic dynamics in the following parabolic-elliptic chemotaxis system with space and time dependent logistic source,
\begin{equation}\label{main-eq-abstract}
\begin{cases}
\partial_tu=\Delta u -\chi\nabla\cdot(u\nabla v)+u(a(x,t)-ub(x,t)),\quad x\in\R^N,\cr
0=\Delta v-\lambda v+\mu u ,\quad x\in\R^N,
\end{cases}
\end{equation}
where $N\ge 1$ is a positive integer,  $\chi, \lambda$ and $\mu$ are positive constants,
 and the functions $a(x,t)$ and $b(x,t)$ are positive and bounded.
In the first of the series, we investigate the persistence and asymptotic spreading in  \eqref{main-eq-abstract}. To this end,
 under some explicit condition on the parameters, we first show that \eqref{main-eq-abstract} has a unique nonnegative time global classical solution
  $(u(x,t;t_0,u_0),v(x,t;t_0,u_0))$ with $u(x,t_0;t_0,u_0)=u_0(x)$ for every $t_0\in\R$ and every  nonnegative bounded and uniformly continuous initial function
  $u_0$. Next we show the  pointwise  persistence phenomena of the solutions  in the sense that, for any solution  $(u(x,t;t_0,u_0),v(x,t;t_0,u_0))$ of \eqref{main-eq-abstract} with strictly positive initial function $u_0$, there are $0<m(u_0)\le M(u_0)<\infty$ such that
  $$
  m(u_0)\le u(x,t+t_0;t_0,u_0)\le M(u_0)\quad \forall\,\, t\ge 0,\,\, x\in\R^N,
  $$
  and show the uniform persistence phenomena of  solutions in the sense that there are $0<m<M$ such that for any strictly positive initial function $u_0$,
  there is $T(u_0)>0$ such that
  $$
  m\le u(x,t+t_0;t_0,u_0)\le M\quad \forall\, t\ge T(u_0),\,\, x\in\R^N.
  $$
   We then discuss  the spreading properties of solutions to \eqref{main-eq-abstract} with compactly supported initial function and prove that there are positive constants $0<c_{-}^{*}(a,b,\chi,\lambda,\mu)\le c_{+}^{*}(a,b,\chi,\lambda,\mu)<\infty$ such that for every $t_0\in \R$ and every nonnegative initial function $u_0\in C^b_{\rm unif}(\R^N)$ with nonempty compact support, we have that
$$
\lim_{t\to\infty}\sup_{|x|\geq ct}u(x,t+t_0;t_0,u_0)=0, \quad \forall c> c_{+}^{*}(a,b,\chi,\lambda,\mu),
$$
and
$$
\liminf_{t\to\infty}\sup_{|x|\leq ct}u(x,t+t_0;t_0,u_0)>0, \quad \forall 0< c< c_{-}^{*}(a,b,\chi,\lambda,\mu).
$$
We also discuss the spreading properties of solutions to \eqref{main-eq-abstract} with front-like initial functions.
In the second and third of the series, we will study the existence, uniqueness, and stability of strictly positive entire solutions of \eqref{main-eq-abstract} and the existence of transition fronts of \eqref{main-eq-abstract}, respectively.
\end{abstract}

\medskip
\noindent{\bf Key words.} Parabolic-elliptic chemotaxis system, logistic source, pointwise persistence, uniform persistence, asymptotic spreading, comparison principle.

\medskip
\noindent {\bf 2010 Mathematics Subject Classification.}  35B35, 35B40, 35K57, 35Q92, 92C17.

\section{Introduction and Statement of the Main Results}

The current series of three papers is devoted to the study of the asymptotic dynamics of the following parabolic-elliptic  chemotaxis system  with space and time dependent logistic source on $\R^N$,
\begin{equation}\label{P}
\begin{cases}
\partial_tu=\Delta u -\chi\nabla\cdot(u\nabla v)+u(a(x,t)-b(x,t)u),\quad x\in\R^N,\cr
0=\Delta v-\lambda v+\mu u ,\quad x\in\R^N,
\end{cases}
\end{equation}
where $u(x,t)$ and $v(x,t)$ denote the population densities of some mobile species and chemical substance, respectively,
 $\chi$ is a positive constant which measures the sensitivity of the mobile species to the chemical substance, $a(x,t)$ and $b(x,t)$ are positive functions  and measure the self growth and self limitation of the mobile species, respectively. The constant $\mu$ is positive and the term $+\mu u$ in the second equation of \eqref{P} indicates that the mobile species produces the chemical substance over time. The positive constant $\lambda$ measures the degradation rate of the chemical substance.  System \eqref{P}  is a space-time logistic source dependant variant of the celebrated parabolic-elliptic Keller-Segel chemotaxis systems (see \cite{KeSe1, KeSe2}). 
  Chemotaxis, the oriented movements of biological cells and organism in response to chemical gradient, plays a very important role in a wide range of biological phenomena (see \cite{ISM04,DAL1991,KJPJAS03}, etc.), and accordingly a considerable literature is concerned with its mathematical analysis.

It is  known that chemotaxis systems present very interesting dynamics. For example, consider the following chemotaxis model,
\begin{equation}\label{general model}
\begin{cases}
\partial_tu=\Delta u -\chi\nabla\cdot(u\nabla v)+u(a(x,t)-b(x,t)u),\quad x\in\Omega,\cr
\tau v_t=\Delta v-\lambda v+\mu u ,\quad x\in\Omega
\end{cases}
\end{equation}
complemented with certain boundary conditions if $\Omega \subset \R^N$ is a bounded domain, where $\tau\geq 0$ is a nonnegative constant link to the speed of diffusion of the chemical substance. Note that when $\tau=0$ and $\Omega=\R^N$ in \eqref{general model}, we recover \eqref{P}. Hence, \eqref{P}  models the situation where the chemoattractant defuses very quickly and the underlying environment is very large. When the functions $a(x,t)$ and $b(x,t)$ are identically equal to zero,  and $\chi>0$ in \eqref{general model}, it is known that finite time blow-up  may occurs if either $N=2$ and the total initial population mass is large enough, or $N\geq 3$ (see \cite{ Dirk and Winkler, KKAS, kuto_PHYSD, NAGAI_SENBA_YOSHIDA, Hil1,  win_jde, win_JMAA_veryweak, win_arxiv}). It is also known that some radial solutions to \eqref{P} in plane collapse into a persistent Dirac-type singularity in the sense that a globally defined measure-valued solution exists which has a singular part beyond some finite time and asymptotically approaches a Dirac measure (see \cite{LSV1,TeWi1}). We refer the reader to \cite{BBTW,HoSt} and the references therein for more insights in the studies of chemotaxis models.

It is also  known that when \eqref{general model} is considered with logistic source, that is $a(x,t)>0$ and $b(x,t)>0$, the finite time blow-up phenomena may be suppressed to some extent. For example, when $\Omega$ is a bounded smooth domain,  $a(x,t)$ and $b(x,t)$ are constant positive functions, $\tau=0$ and  $\lambda=\mu=1$, it is shown in \cite{TeWi2} that if either  $N\leq 2$ or $b>\frac{N-2}{N}\chi$, then  for every nonnegative H\"older's continuous initial $u_0(x)$, \eqref{general model} with Neumann boundary condition possesses a unique  bounded global classical solution $(u(x,t;u_0),v(x,t;u_0))$ with $u(x,0;u_0)=u_0$. Furthermore, if $b>2\chi$, then the trivial steady state $(\frac{a}{b},\frac{a}{b})$ is asymptotically stable with respect to nonnegative and non-identically zero perturbations. These results are extended by the authors of the current paper,  \cite{SaSh1}, to \eqref{P} on $\R^N$ when $a(x,t)$ and $b(x,t)$ are constant functions. The work \cite{SaSh1} also studied some spreading properties of solutions to \eqref{P} with compactly supported initial functions. When $a(x,t)$ and $b(x,t)$ are constant positive functions, $\tau=1$ and  $\lambda=\mu=1$, it is shown in \cite{Win} that it is enough for  $\frac{b}{\chi}$ to be sufficiently large to prevent finite time blow up of classical solutions and to guarantee the stability of the constant equilibrium solution $(\frac{a}{b},\frac{a}{b})$.


In reality, the environments of many living organisms are spatially and temporally heterogeneous. It is of both biological and mathematical interests to study chemotaxis models with certain time and space dependence.
In the case that the chemotaxis is absent
(i.e. $\chi=0$) in \eqref{general model}, the population density $u(x,t)$ of the mobile species
satisfies the following  scalar reaction diffusion equation,
\begin{equation}\label{KPP-Fisher equation}
\partial_t u=\Delta u + u(a(x,t)-b(x,t)u), \  x\in \Omega
\end{equation}
complemented with certain boundary conditions if $\Omega\subset \R^N$ is a bounded domain.
Equation \eqref{KPP-Fisher equation} is called Fisher or KPP type equation in literature because of the pioneering works by Fisher (\cite{Fisher}) and Kolmogorov, Petrowsky, Piskunov
(\cite{KPP}) in the special case $a(t,x)=b(t,x)=1$. A huge amount of research has been carried out toward the asymptotic dynamics of \eqref{KPP-Fisher equation}, see, for example , \cite{CaCo, HeWe, NMN00, ShYi, QXZ}, etc.  for the asymptotic dynamics of \eqref{KPP-Fisher equation}
on bounded domains, and \cite{Berestycki1, BeHaNa1, BeHaNa2, Henri1, Fre, FrGa, LiZh, LiZh1, Nad, NoRuXi, NoXi1, She1, She2, Wei1, Wei2, Zla}, etc.
for the asymptotic dynamics of \eqref{KPP-Fisher equation} on unbounded domains.
In the very recent work \cite{ITBWS16}, the authors studied the asymptotic dynamics of \eqref{general model} on bounded domain $\Omega$ with Neumann boundary condition and with space and time  dependent logistic source. However, there is little study on the asymptotic dynamics of \eqref{general model} on unbounded domain with space and time dependent logistic source.

The objective of the  current series of three papers is to carry out a systematic study of the asymptotic dynamics of the chemotaxis model \eqref{P}
on the whole space
with $a$ and $b$ being depending on both $x$ and $t$.
In this first part of the series, we investigate the global existence of nonnegative classical solutions of \eqref{P},
the persistence of classical solutions of \eqref{P} with strictly positive initial functions, and the asymptotic spreading properties of classical solutions of   \eqref{P} with compact supported and front-like nonnegative initial functions. In the second and third of the series, we will study the existence, uniqueness, and stability of strictly positive entire solutions of \eqref{P} and the existence of transition fronts of \eqref{P}, respectively.

In the rest of the introduction, we introduce the notations and standing assumptions, and state the main results of this paper.


\subsection{Notations and standing assumptions}

For every $x=(x_1,x_2,\cdots,x_N)\in\R^N$, let $|x|_{\infty}=\max\{|x_i|, \ |\ i=1,\cdots,N\}$ and $|x|=\sqrt{|x_1|^2+\cdots+|x_N|^2}$. We define
$$S^{N-1}:=\{x\in\R^N\ |\ |x|=1\}.
$$
For $x=(x_1,x_2,\cdots,x_N),y=(y_1,y_2,\cdots,y_N)\in\R^N$, $x\cdot y=\sum_{i=1}^N x_iy_i$. For every $x\in\R^N$ and $r>0$ we define
$$B(x,r):=\{y\in\R^N\ |\ |x-y|<r\}.$$

  Let
$$
C_{\rm unif}^b(\R)=\{u\in C(\R)\,|\, u(x)\quad \text{is uniformly continuous in}\,\,\, x\in\R\quad \text{and}\,\, \sup_{x\in\R}|u(x)|<\infty\}
$$
equipped with the norm $\|u\|_\infty:=\sup_{x\in\R}|u(x)|$. For any $0\le \nu<1$, let
$$
C_{\rm unif}^{b,\nu}(\R^N):=\{u\in C_{\rm unif}^b(\R^N)\,|\, \sup_{x,y\in\R,x\not = y}\frac{|u(x)-u(y)|}{|x-y|^\nu}<\infty\}
$$
with norm $\|u\|_{C_{\rm unif}^{b,\nu}}:=\sup_{x\in\R}|u(x)|+\sup_{x,y\in\R,x\not =y}\frac{|u(x)-u(y)|}{|x-y|^\nu}$. Hence $C_{\rm unif}^{b,0}(\R^N)=C_{\rm unif}^{b}(\R^N)$.

 For every function $w : \R^n\times I\to \R$, where $I\subset \R$, we set $w_{\inf}(t):=\inf_{x\in\R^N}w(x,t)$, $w_{\sup}(t):=\sup_{x\in\R^N}w(x,t)$, $w_{\inf}=\inf_{x\in\R^N,t\in I}w(x,t)$ and $w_{\sup}=\sup_{x\in\R^N,t\in I}w(x,t)$.  In particular for every $u_0\in C^{b}_{\rm unif}(\R^n)$, we set $u_{0\inf}=\inf_{x}u_0(x) $ and  $u_{0\sup}=\sup_{x}u_0(x)$.

A function  $u_0\in C^b_{\rm unif}(\R^N)$ is called {\it nonnegative} if $u_{0\inf}\ge 0$, {\it positive} if $u_0(x)>0$ for all $x\in\R^N$,
and {\it strictly positive} if $u_{0\inf}>0$. For given $\xi\in S^{N-1}$, a nonnegative function $u_0\in C_{\rm unif}^b(\R^N)$ is call a {\it front-like function in the direction of $\xi$} if $\liminf_{x\cdot\xi \to -\infty}u_0(x)>0$ and $u_0(x)=0$ for $x\in\R^N$ with $x\cdot \xi\gg 1$.

In what follows we shall always suppose that the following hypothesis holds:

\medskip

\noindent {\bf (H)} {\it $a(t,x)$ and  $b(t,x)$ are uniformly H\"older continuous in $(x,t)\in\R^N\times\R$ with exponent $0<\nu<1$ 
and
$$
 0<\inf_{t\in\R,x\in\R^n}\min\{a(x,t), b(x,t)\} \leq \sup_{t\in\R,x\in \R^n}\max\{a(x,t),b(x,t)\}<\infty.
$$
}

\subsection{Statements of  the main results}

The objective of the current part of the series is to investigate the global existence and  persistence of nonnegative bounded classical solutions
of \eqref{P},
 and the spreading properties of nonnegative classical solutions of  \eqref{P} with compactly supported or front-like initial functions.

    We say that $(u(x,t),v(x,t))$ is a {\it classical solution } of \eqref{P} on $[t_0, T)$ if
 $(u(\cdot,\cdot),v(\cdot,\cdot))\in C(\R^N\times[t_0,T))\cap C^{2,1}(\R^N\times(t_0,T))$
    and satisfies \eqref{P} for $(x,t)\in\R^N\times(t_0,T)$ in the classical sense.
    When  a classical solution  $(u(x,t),v(x,t))$ of \eqref{P} on $[t_0,T)$ satisfies
    $u(x,t)\geq 0$ and $v(x,t)\geq 0$ for every $(x,t)\in\R^N\times[t_0,T)$, we say that it is nonnegative. A {\it global classical solution}  of \eqref{P} on $[t_0,\infty)$ is a classical solution on $[t_0, T)$ for every $T>0$. We say that $(u(x,t),v(x,t))$ is an
   {\it entire solution} of \eqref{P} if $(u(x,t),v(x,t))$ is a global classical solution of \eqref{P} on $[t_0, \infty)$ for every $t_0\in\R.$
     For given  $u_0\in C_{\rm unif}^b(\R^N)$ and $t_0,T\in\R$ with $T>t_0$, if $(u(x,t),v(x,t))$ is a classical solution of \eqref{P} on $[t_0, T)$ with
     $u(x,t_0)=u_0(x)$ for all $x\in\R$, we
     denote it as
    $(u(x,t;t_0,u_0),v(x,t;t_0,u_0))$ and call it the {\it solution of \eqref{P} with initial function
    $u_0(x)$ at time $t_0$}.

Note that, due to biological interpretations, only nonnegative initial functions will be of interest. Note also that for $u_0\equiv 0$,
$(u(x,t;t_0,u_0),v(x,t;t_0,u_0))\equiv (0,0)$ for all $t\in\R$ and $x\in\R^N$.
From both mathematical and biological point of view, it is important to find conditions which guarantee the global existence of
$(u(x,t;t_0;u_0),v(x,t;t_0,u_0))$ for every $t_0\in\R$ and $u_0\in C_{\rm unif}^b(\R^N)\setminus\{0\}$ with $u_0\ge 0$.
We have the following result on the global  existence of  classical solutions $(u(x,t;t_0,u_0),v(x,t;t_0,u_0))$ of \eqref{P} with
$u_0\ge 0$ and $u_0\not\equiv 0$.

\begin{tm}[Global existence]\label{global-existence-tm}
Suppose that $\chi\mu \leq b_{\inf},$
then for every $t_0\in\R$ and nonnegative function $u_0\in C^{b}_{\rm unif}(\R^n)\setminus\{0\}$, \eqref{P} has a unique nonnegative global classical solution $(u(x,t;t_0,u_0)$, $v(x,t;t_0,u_0))$
satisfying
$$
\lim_{t\searrow 0}\|u(\cdot,t_0+t;t_0,u_0)-u_0\|_{\infty}=0.
$$ Moreover, it holds that
\begin{equation}\label{u-upper-bound-eq1}
\|u(\cdot,t+t_0;t_0,u_0)\|_{\infty}\leq \|u_0\|_{\infty}e^{ a_{\sup} t}.
\end{equation}
Furthermore, if $${\bf (H1):}\quad b_{\inf}>\chi\mu$$  holds, then the following hold.
\begin{description}
\item[(i)] For every nonnegative initial function $  u_0\in C^{b}_{\rm unif}(\R^N)\setminus\{0\}$ and $t_0\in\R$, there hold
\begin{equation}\label{u-upper-bound-eq2}
0\leq u(x,t+t_0;t_0,u_0)\leq \max\{ \|u_0\|_{\infty}, \frac{a_{\sup}}{b_{\inf}-\chi\mu}\}\,\, \forall\,\, t\ge 0,\,\, \forall\, x\in\R^N,
\end{equation}
and
\begin{equation}\label{u-upper-bound-eq3}
 \limsup_{t\to\infty} \|u(\cdot,t+t_0;t_0,u_0)\|_{\infty}\leq \frac{a_{\sup}}{b_{\inf}-\chi\mu}.
\end{equation}

\item[(ii)]  For every $u_0\in C_{\rm unif}^b(\R^N)$ with
$\inf_{x\in\R^N}u_0(x)>0$ and $t_0\in\R$, there hold
\begin{equation}
\label{asymptotic-lower-bound}
\frac{a_{\inf}}{b_{\sup}}\leq\limsup_{t\to\infty} \sup_{x\in\R^N}u(x,t+t_0;t_0,u_0),\quad \liminf_{t\to\infty}\inf_{x\in\R^N}u(x,t+t_0;t_0,u_0)\le \frac{a_{\sup}}{b_{\inf}}.
\end{equation}

\item[(iii)] For every positive real number $M>0$, there  is a constant $K_{1}=K_1(\nu,M,a,b)$ such that for every $  u_0\in C^{b}_{\rm unif}(\R^N)$ with $0\leq u_0\leq M$, we have
\begin{equation}\label{uniform-holder-bound-for-v}
\|v(\cdot,t+t_0;t_0,u_0)\|_{C^{1,\nu}_{\rm unif}(\R^N)}\leq K_1, \quad \forall \ t_0\in\R,\ \forall\ t\geq 0.
\end{equation}
\end{description}
\end{tm}

\begin{rk}
\label{thm1-rk}
\begin{description}
\item[(i)] Theorem \ref{global-existence-tm} rules out finite-time blow-up phenomena when $\chi\mu\leq b_{inf}$, and guarantees that classical solutions are bounded whenever {\bf (H1)} holds. This results recovers Theorem 1.5 in \cite{SaSh1} when $a(x,t)$
and $b(x,t)$ are constant functions.
\item[(ii)]  Theorem \ref{global-existence-tm} (iii) provides a uniform upper bound independent of the initial time $t_0$, for $\|v(\cdot,t+t_0;t_0;u_0)\|_{C^{1,\nu}_{\rm unif}(\R^N)}$ for every $u_0\in C^{b}_{\rm unif}(\R^N)$ whenever {\bf (H1)} holds. This result will be useful to show that a sequence of solutions converges up to a subsequence in the open compact topology to a solution.

\end{description}
\end{rk}

 By Theorem \ref{global-existence-tm}, for any strictly positive $u_0\in C_{\rm unif}^b(\R^N)$,
 $\limsup_{t\to\infty}\sup_{x\in\R^N}u(x,t+t_0;t_0,u_0) $ has a positive lower bound and $\liminf_{t\to\infty}\inf_{x\in\R^N}u(x,t+t_0;t_0,u_0)$ has
 a positive upper bound.
 But it is not clear whether there is a positive lower bound (resp. a positive lower bound independent of $u_0$) for $\liminf_{t\to\infty}\inf_{x\in\R^N}u(x,t+t_0;t_0,u_0)$
 under hypothesis {\bf (H1)}, which is related to the persistence in \eqref{P}.

Persistence is an important concept in population models. Assume {\bf (H1)}.
We say {\it pointwise persistence} occurs in \eqref{P} if for any strictly positive $u_0\in C_{\rm unif}^b(\mathbb{R}^N)$,
there { exist positive real numbers} $m(u_0)>0$ { and $M(u_0)>0$} such that
 \begin{equation}
 \label{persistence-eq1-0}
 m(u_0)\le u(x,t+t_0;t_0,u_0)\le  M(u_0)\,\, \forall\,\, x\in\R^N,\,\, t_0\in\mathbb{R}\,\, {\rm and}\,\, t\ge 0.
 \end{equation}
 It is said that {\it uniform persistence} occurs in \eqref{P} if there are $0<m<M<\infty$ such that for any  $t_0\in\R$ and any strictly positive initial function $u_{0}\in C^{b}_{\rm unif}(\R^N)$, there exists $T(u_0)>0$ such that
 $$
 m\le u(x,t+t_0;t_0,u_0)\le M\quad \forall\,\, t\ge { T(u_0),\,\,\,x\in\R^N, \,\, \,\, t_0\in\R.}
 $$
 Note that uniform persistence implies pointwise persistence.

We have the following results on the pointwise persistence/uniform persistence of solutions of \eqref{P} with strictly positive initials.


\begin{tm}
\label{Main-thm1}
\begin{itemize}
\item[(i)] (Pointwsie persistence) Suppose that {\bf (H1)} holds. Then pointwise persistence occurs in \eqref{P}.

 \item[(ii)] (Uniform persistence) 
  Suppose that {\bf (H1)} holds.
 If, furthermore,  $$
{\bf (H2)}: \quad  b_{\inf}>(1+\frac{a_{\sup}}{a_{\inf}})\chi\mu
$$ holds,  then uniform persistence occurs in \eqref{P}. In particular, for every strictly positive initial $u_0\in C^b_{\rm unif}(\R^N)$
  and $\varepsilon>0$, there is $T_{\varepsilon}(u_0)>0$ such  that the unique classical solution $(u(x,t+t_0;t_0,u_0), v(x,t+t_0;t_0,u_0))$ of \eqref{P} with $u(\cdot,t_0;t_0,u_0)=u_0(\cdot)$ satisfies
\begin{equation}\label{attracting-rect-eq1}
\underline{M}-\varepsilon\leq u(x,t+t_0;t_0,u_0)\leq \overline{M}+\varepsilon, \quad \forall t\geq T_{\varepsilon}(u_0), \ x\in\R^N,\,\,\, t_{0}\in\R
\end{equation} and
\begin{equation}\label{attracting-rect-eq1'}
\frac{\mu\underline{M}}{\lambda}-\varepsilon\leq v(x,t+t_0;t_0,u_0)\leq \frac{\mu\overline{M}}{\lambda}+\varepsilon, \quad \forall t\geq T_{\varepsilon}(u_0), \ x\in\R^N,\,\,\, t_{0}\in\R,
\end{equation}
where
\begin{equation}\label{attracting-rect-eq2}
\underline{M}:=\frac{(b_{\inf}-\chi\mu)a_{\inf}-\chi\mu a_{\sup} }{(b_{\sup}-\chi\mu)(b_{\inf}-\chi\mu)-(\chi\mu)^2}>\frac{a_{\inf}-\frac{\chi\mu a_{\sup}}{b_{\inf}-\chi\mu} }{b_{\sup}-\chi\mu}
\end{equation}
and
\begin{equation}\label{attracting-rect-eq3}
\overline{M}:=\frac{(b_{\sup}-\chi\mu)a_{\sup}-\chi\mu a_{\inf} }{(b_{\sup}-\chi\mu)(b_{\inf}-\chi\mu)-(\chi\mu)^2}< \frac{a_{\sup}}{b_{\inf}-\chi\mu}.
\end{equation}
Furthermore, the set
\begin{equation} \label{Invariant set} \mathbb{I}_{inv}:=\{ u\in C^b_{\rm unif}(\R^N)\ |\ \underline{M}\leq u_{0}(x)\leq \overline{M}, \ \forall\, x\in\R^N\}
\end{equation} is a positively invariant set for solutions of \eqref{P} in the sense that for every $t_0\in\R$ and $u_0\in\mathbb{I}_{inv}$, we have that $u(\cdot,t+t_0;t_0,u_0)\in\mathbb{I}_{inv}$ for every $t\geq0$.
\end{itemize}
\end{tm}

\begin{rk}
\begin{itemize}
\item[(1)]  Assume {\bf (H1)}. By Theorem \ref{global-existence-tm} (ii), $m(u_0)\leq \frac{a_{\sup}}{b_{\inf}} $ for every $u_0\in C^b_{\rm unif}(\R^N)$ with $u_{0\inf}>0$. It remains open whether uniform persistence occurs under {\bf (H1)}. It should be noted that uniform pointwise occurs under {\bf (H1)} when \eqref{P} is studied on bounded domains with Neumann boundary conditions (see \cite[Remark 1.2]{Issa-Shen}).


\item[(2)] The proof of Theorem \ref{Main-thm1} (i) is highly nontrivial and is based on a key and  fundamental result proved in Lemma \ref{main-new-lm2}.  Roughly speaking, Lemma \ref{main-new-lm2} shows that for any given time $T>0$, the concentration $u(x,t;t_0,u_0)$ of the mobile species at time $t_0+T$  can be controlled
     below by $u_{0\inf}$  provided that $u_{0\inf}$ is sufficiently small. This result will also play a  crucial role in  the study of existence of strictly positive entire solutions in the second part of the series.

\item[(3)] When the functions $a(x,t)$ and $b(x,t)$ are both space and time homogeneous, $\underline{M}=\overline{M}=\frac{a}{b}$, where $\underline{M}$ and $\overline{M}$ are as in Theorem \ref{Main-thm1} (ii). Hence we recover \cite[Theorem 1.8]{SaSh1}. 

\item[(4)] Assume {\bf (H1)}. By Theorems \ref{global-existence-tm} and \ref{Main-thm1}, for any $t_0\in\R$ and strictly positive $u_0\in C_{\rm unif}^b(\R^N)$,
$$
0<\liminf_{t\to\infty}\inf_{x\in\R^N}u(x,t+t_0;t_0,u_0)\le\limsup_{t\to\infty}\sup_{x\in\R^N}u(x,t+t_0;t_0,u_0)<\infty.
$$
Naturally, it is important to know whether there is a {\it strictly positive entire solution}, that is,
an entire solution $(u^+(x,t),v^+(x,t))$ of \eqref{P} with $\inf_{t\in\R,x\in\R^N}u^+(x,t)>0$.
It is also important to know the stability of strictly  entire positive solutions of \eqref{P} (if exist) and to investigate the asymptotic behavior
of globally defined classical solutions with strictly positive initial functions.
These problems will be studied in the second part of the series.

\end{itemize}
\end{rk}

 By Theorem \ref{Main-thm1}, for any strictly positive $u_0\in C_{\rm unif}^b(\R^N)$,
there is a positive lower bound (resp. a positive lower bound independent of $u_0$) for $\liminf_{t\to\infty}\inf_{x\in\R^N}u(x,t+t_0;t_0,u_0)$  provided (H1)  (resp. (H2)) holds. It is clear that for any nonnegative $u_0\in C_{\rm inf}^b(\R^N)$ with nonempty and compact support,
$\inf_{x\in\R^N} u(x,t+t_0;t_0,u_0)=0$ for any $t\ge 0$ and hence $\liminf_{t\to\infty}\inf_{x\in\R^N}u(x,t+t_0;t_0,u_0)=0$. It is important to know whether $\limsup_{t\to\infty}\sup_{x\in\R^N}u(x,t+t_0;t_0,u_0)$ has a positive lower bound;  whether there is a positive number $m(u_0)>0$ such that
the set $I(t):=\{x\in\R^N\,|\, u(x,t+t_0;t_0,u_0)\ge m(u_0)\}$ spreads into the whole space $\R^N$ as $t\to\infty$; and if so, how fast the set $I(t)$ spreads into the whole space $\R^N$, which is related to the asymptotic spreading property in \eqref{P}.

 We have the following two theorems on the asymptotic behavior or  spreading properties of the solutions of \eqref{P} with compactly supported and front-like initial functions, respectively.

\begin{tm}[Asymptotic spreading]\label{spreading-properties}
\begin{itemize}
\item[(1)] Suppose that {\bf (H1)} holds. Then for every $t_0\in \R$ and every nonnegative initial function $u_0\in C^b_{\rm unif}(\R^N)$ with nonempty compact support $supp(u_0)$, we have that
\begin{equation*}
\lim_{t\to\infty}\sup_{|x|\geq ct}u(x,t+t_0;t_0,u_0)=0, \quad \forall c> c_{+}^{*}(a,b,\chi,\lambda,\mu),
\end{equation*}
where
\begin{equation}
\label{c-positive-star}
c_{+}^{*}(a,b,\chi,\lambda,\mu):=2\sqrt{a_{\sup}}+ \frac{\chi\mu\sqrt{N}a_{\sup}}{2(b_{\inf}-\chi\mu)\sqrt{\lambda}}.
\end{equation}

\item[(2)]
Suppose that
\begin{equation*}
{\bf (H3)}: \ b_{\inf}>\left(1+\frac{\Big(1+\sqrt{1+\frac{Na_{\inf}}{4\lambda}}\Big)a_{\sup}}{2a_{\inf}}\right)\chi\mu.
\end{equation*}
Then for every $t_0\in\R$ and every nonnegative initial function $u_0\in C^b_{\rm unif}(\R^N)$ with nonempty support $supp(u_0)$, we have that
\begin{equation*}\label{lower-bound-spreading-speed}
\liminf_{t\to\infty}\inf_{|x|\leq ct}u(x,t+t_0;t_0,u_0)>0, \quad \forall\,  0< c< c_{-}^{*}(a,b,\chi,\lambda,\mu),
\end{equation*}
where
\begin{equation}
\label{c-negative-star}
c_{-}^*(a,b,\chi,\lambda,\mu):=2\sqrt{a_{\inf}-\frac{\chi\mu a_{\sup}}{b_{\inf}-\chi\mu}}-\chi\frac{\mu\sqrt{N}a_{\sup}}{2\sqrt{\lambda}(b_{\inf}-\chi\mu)}.
\end{equation}
\end{itemize}
\end{tm}

\medskip

\begin{tm}[Asymptotic spreading]\label{spreading-properties-1}
\begin{itemize}
\item[(1)] Suppose that {\bf (H1)} holds. Then for every $\xi\in S^{N-1}$, every $t_0\in \R$, and every nonnegative front-like initial function $u_0\in C^b_{\rm unif}(\R^N)$ in the direction of $\xi$,   we have that
\begin{equation*}
\lim_{t\to\infty}\sup_{x\cdot \xi\geq ct}u(x,t+t_0;t_0,u_0)=0, \quad \forall c> c_{+}^{*}(a,b,\chi,\lambda,\mu),
\end{equation*}
where $c_{+}^{*}(a,b,\chi,\lambda,\mu)$ is as in \eqref{c-positive-star}.

\item[(2)]
Suppose that {\bf (H3)} holds.
Then for every $\xi\in S^{N-1}$,  every $t_0\in\R$, and every nonnegative  front-like initial function $u_0\in C^b_{\rm unif}(\R^N)$ in the direction of $\xi$, we have that
\begin{equation*}\label{lower-bound-spreading-speed}
\liminf_{t\to\infty}\inf_{x\cdot\xi\leq ct}u(x,t+t_0;t_0,u_0)>0, \quad \forall\, 0< c< c_{-}^{*}(a,b,\chi,\lambda,\mu),
\end{equation*}
where $c_{-}^*(a,b,\chi,\lambda,\mu)$ is as in \eqref{c-negative-star}.
\end{itemize}
\end{tm}

\begin{rk}
\begin{itemize}
\item[(1)]{ We first note that under hypothesis {\bf (H3)}, using the uniform pointwise result given by Theorem \ref{Main-thm1} (ii), one can show that for every $t_0\in \R$ and every nonnegative  $u_0\in C^b_{\rm unif}(\R^N)$ with nonempty compact support, there hold
$$
\underline{M}\leq \liminf_{t\to \infty}\inf_{|x|\leq ct}u(x,t+t_0;t_0,u_0)\quad {\rm and}\quad   \limsup_{t\to\infty}\sup_{|x|\leq ct}u(x,t+t_0;t_0,u_0)\leq \overline{M}$$
for any $ 0<c<c_{+}^{*}(a,b,\chi,\lambda,\mu)$,
where $\underline{M}$, $\overline{M}$, and $c_{+}^{*}(a,b,\chi,\lambda,\mu) $ are given by \eqref{attracting-rect-eq2}, \eqref{attracting-rect-eq3}, and \eqref{c-negative-star} respectively. Hence, in the case of space-time homogeneous logistic, solutions of \eqref{P} with nonnegative initial data spread to the unique constant equilibrium solution, which recover [ Theorem 1.9 (1) \cite{SaSh1}] and [Theorem D (ii) \cite{SaSh2}].
}

\item[(2)] Note that $\lim_{\chi\to 0^+}(c_{-}^*(a,b,\chi,\lambda,\mu),c_{+}^*(a,b,\chi,\lambda,\mu) )=(2\sqrt{a_{\inf}},2\sqrt{a_{\sup}})$. Thus in the case of space-time homogeneous logistic source, we obtain that $c_{-}^*(a,b,\chi,\lambda,\mu) $  and $c_{-}^*(a,b,\chi,\lambda,\mu)$ converge to $2\sqrt{a}$, the minimal spreading speed of \eqref{KPP-Fisher equation}.

\item[(3)] Let
$$
C_0^b(\R^N)=\{u\in C_{\rm unif}^b(\R^N)\,|\, u\ge 0, \,\,\, {\rm supp}(u)\,\, {\rm is\,\,\,  nonempty\,\,\,  and\,\,\,  compact}\},
$$
$$
C_{\rm sup}=\{c\,|\,
\limsup_{t\to\infty}\sup_{|x|\geq c^{'}t}u(x,t+t_0;t_0,u_0)=0, \quad \forall c^{'}> c, \,\, u_0\in C_0^b(\R^N)\},
$$
and
$$
C_{\rm inf}=\{c\,|\, \liminf_{t\to\infty}\inf_{|x|\le c^{''}t}u(x,t+t_0;t_0,u_0)>0, \quad \forall c^{'}<c, \,\, u_0\in C_0^b(\R^N)\}.
$$
Let
$$
c_{\rm inf}^*(a,b,\chi,\lambda,\mu)=\sup\{c\,|\, c\in C_{\rm inf}\},\quad c_{\rm sup}^*(a,b,\chi,\lambda,\mu)=\inf\{c\,|\, c\in C_{\rm sup}\}.
$$
$[c_{\rm inf}^*(a,b,\chi,\lambda,\mu), c_{\rm sup}^*(a,b,\chi,\lambda,\mu)]$ is called the {\rm spreading speed interval} of \eqref{P}.
By Theorem \ref{spreading-properties},
$$
c_-^*(a,b,\chi,\lambda,\mu)\le c_{\rm inf}^*(a,b,\chi,\lambda,\mu)\le c_{\rm sup}^*(a,b,\chi,\lambda,\mu)\le c_+^*(a,b,\chi,\lambda,\mu).
$$
It is interesting to know how $c_{\rm inf}^*(a,b,\chi,\lambda,\mu)$ and $c_{\rm sup}^*(a,b,\chi,\lambda,\mu)$ depend on $\chi$, in particular,
it is interesting to know whether the chemotaxis speeds up, or slows down, or neither speeds up nor slows down the spreading speeds of the mobile species.
We plan to study this issue somewhere else.

\item[(4)] Theorem \ref{spreading-properties-1} describes the spreading properties of nonnegative solutions with front-like initial functions.
It is interesting to know whether \eqref{P} admits transition front solutions, which are the analogue of traveling wave solutions in
the space and time homogeneous chemotaxis models and have recently been widely studied in the absence of chemotaxis. We will study the existence
of transition fronts in \eqref{P} in the third part of the series.
\end{itemize}
\end{rk}

The rest of the paper is organized as follows. In section 2, we study the existence of global classical solutions and prove Theorem \ref{global-existence-tm}. Section 3 is devoted to the investigation of pointwise and uniform persistence in \eqref{P}.  Theorem \ref{Main-thm1} is proved in this section. In section 4, we discuss the asymptotic properties of solutions and give the proofs of Theorems \ref{spreading-properties} and  \ref{spreading-properties-1}.

\section{Global Existence}

This section is devoted to the study of the global existence of classical solutions to \eqref{P}. We start with the following result about the local existence and uniqueness of classical solutions.

\begin{lem}[Local existence]\label{local-existence-lem}
For every $t_0\in\R$ and nonnegative function $u_0\in C^{b}_{\rm unif}(\R^n)$, there is $T_{\max}=T_{\max}(t_0,u_0)\in (0, \infty]$ such that \eqref{P} has a unique nonnegative classical solution $(u(x,t;t_0,u_0),v(x,t;t_0,u_0))$ on $[t_0 , t_0+T_{\max})$
satisfying
\begin{equation}\label{local-cont-at-0-eq}
\lim_{t\searrow 0}\|u(\cdot,t_0+t;t_0,u_0)-u_0\|_{\infty}=0.
\end{equation}
Furthermore, if $T_{\max}<\infty$, then
\begin{equation}\label{extension-criterion}
\lim_{t\nearrow T_{\max}}\|u(\cdot,t_0+t;t_0,u_0)\|_{\infty}=\infty.
\end{equation}
\end{lem}
\begin{proof}
The result follows from properly modified arguments of the proof of \cite[Theorem 1.1]{SaSh1}.
\end{proof}

Next, we present a lemma on the bounds of $\|v(\cdot,t+t_0;t_0,u_0)\|_\infty$ and $\|\nabla v(\cdot,t+t_0;t_0,u_0)\|_\infty$, which will be used in this section as well as later sections.

\begin{lem}
\label{bounds-on-v}
For every $t_0\in\R$ and nonnegative function $u_0\in C^{b}_{\rm unif}(\R^n)$, the following hold,
\begin{equation}\label{v-bound-eq1}
\|v(\cdot, t+t_0;t_0,u_0)\|_{\infty}\leq \frac{\mu}{\lambda}\|u(\cdot, t+t_0;t_0,u_0)\|_{\infty},\ \forall\,\, 0\le t<T_{\max}(t_0,u_0),
\end{equation}
and
\begin{equation}\label{v-bound-eq2}
\|\nabla v(\cdot, t+t_0;t_0,u_0)\|_{\infty}\leq \frac{\mu\sqrt{N}}{\sqrt{\lambda}}\|u(\cdot, t+t_0;t_0,u_0)\|_{\infty},\ \forall\, \, 0\le t<T_{\max}(t_0,u_0).
\end{equation}
\end{lem}

\begin{proof}
Observe that
$$
v(x, t+t_0;t_0,u_0)=\mu\int_{0}^{\infty}\int_{\R^N}\frac{e^{-\lambda s}e^{-\frac{|x-y|^2}{4s}}}{(4\pi s)^{\frac{N}{2}}}u(y, t+t_0;t_0,u_0)dyds.
$$
\eqref{v-bound-eq1} and \eqref{v-bound-eq2} then follow from simple calculations.
\end{proof}

Now we present the proof of Theorem \ref{global-existence-tm}. Note that Theorem \ref{global-existence-tm} provides a sufficient condition on the parameters $\chi$ and $b_{\inf}$ to guarantee the existence of time globally defined classical solutions.

\begin{proof}[Proof of Theorem \ref{global-existence-tm}]
Let $t_0\in\R$ and $u_0\in C^{b}_{\rm unif}(\R^n)$, $u_0\geq 0$, be given. According to Lemma \ref{local-existence-lem}, there is $T_{\max}=T_{\max}(t_0,u_0)\in (0, \infty]$ such that \eqref{P} has a unique nonnegative classical solution $(u(x,t;t_0,u_0),v(x,t;t_0,u_0))$ on $[t_0 , t_0+T_{\max})$ satisfying \eqref{local-cont-at-0-eq} and \eqref{extension-criterion}.  Since $b_{\inf}\ge \chi\mu$, we have that  $(u(x,t;t_0,u_0),v(x,t;t_0,u_0))$ satisfies
\begin{align}\label{global-eq01}
u_t&=\Delta u -\chi\nabla v\cdot\nabla u +u(a(x,t)-u(b(x,t)-\chi\mu)-\chi\lambda v)\nonumber\\
&\leq \Delta u -\chi\nabla v\cdot\nabla u +u(a(x,t)-u(b(x,t)-\chi\mu))\nonumber\\
&\leq \Delta u -\chi\nabla v\cdot\nabla u +u(a_{\sup} -(b_{\inf} -\chi\mu) u)
\end{align}
for $t\in (t_0,t_0+T_{\max})$.
Thus,  by comparison principles for parabolic equations, it follows from \eqref{global-eq01} that
\begin{equation}\label{global-eq02}
u(x,t+t_0;t_0,u_0)\leq \overline{u}(t;\|u_0\|_{\infty}),\quad \forall \, 0\le t<T_{\max}(t_0,u_0), \ \forall\ x\in\R^N,
\end{equation}
where $\overline{u}(t;\|u_0\|_{\infty})$ solves the ODE
\begin{equation}\label{global-eq03}
\begin{cases}
\frac{d}{dt}\overline{u}=\overline{u}(a_{\sup}-(b_{\inf}-\chi\mu)\overline{u})\\
\overline{u}(0)=\|u_0\|_{\infty}.
\end{cases}
\end{equation}
Since $b_{\inf}\geq \chi\mu$, then $\overline{u}(t;\|u_0\|_{\infty})$ is defined for all $t\geq 0$. This implies that $T_{\max}(t_0,u_0)=\infty$.
 Moreover,  $\overline{u}(t;\|u_0\|_{\infty})\le |u_0\|_{\infty}e^{ta_{\sup}}$ for all $t>0$.
 Hence \eqref{u-upper-bound-eq1} holds.

 \smallskip

 (i) If $b_{\inf}> \chi\mu$, we have that $\overline{u}(t;\|u_0\|_{\infty})\leq\max\{\|u_0\|_{\infty}, \frac{a_{\sup}}{b_{\inf}-\chi\mu}\}$ for all $t>0$, and
 $\lim_{t\to\infty}\overline{u}(t;\|u_0\|_{\infty})=\frac{a_{\sup}}{b_{\inf}-\chi\mu}$ provided that $\|u_0\|_\infty>0$.  Hence \eqref{u-upper-bound-eq2}
 and \eqref{u-upper-bound-eq3} hold.

(ii)  First, by \eqref{v-bound-eq1},
\begin{align}\label{d-001}
u_t&=\Delta u -\chi\nabla v\cdot\nabla u +u(a(x,t)-u(b(x,t)-\chi\mu)-\chi\lambda v)\nonumber\\
&\ge \Delta u -\chi\nabla v\cdot\nabla u +u(a_{\inf}-\|u(\cdot,t;t_0,u_0)\|_\infty(b_{\sup}-\chi\mu)-\chi\lambda \frac{\mu}{\lambda}\|u(\cdot,t;t_0,u_0)\|_\infty)\nonumber\\
&=\Delta u -\chi\nabla v\cdot\nabla u +u(a_{\inf}-\|u(\cdot,t;t_0,u_0)\|_\infty b_{\sup})
\end{align}
for $t>t_0$.
By comparison principle for parabolic equations, we have
$$
u(x,t+t_0;t_0,u_0)\ge e^{\int_{t_0}^{t+t_0}(a_{\inf}-\|u(\cdot,s+t_0;t_0,u_0)\|_\infty b_{\sup})ds}u_{0\inf}\quad \forall t\ge t_0.
$$
This together with $u_{0\inf}>0$  implies that
$$
\inf_{x\in\R^N} u(x,t+t_0;t_0,u_0)>0\quad \forall\,\, t\ge t_0.
$$

 Next, for any $\epsilon>0$, there is $T^\epsilon>0$ such that
$$
u(x,t+t_0;t_0,u_0)\le u^\infty+\epsilon\,\,\, {\rm and}\,\,\, v(x,t+t_0;t_0,u_0)\le \frac{\mu}{\lambda}(u^\infty+\epsilon)\quad \forall \,\, t\ge T^\epsilon,
$$
where  $u^\infty=\limsup_{t\to\infty}\sup_{x\in\R^N}u(x,t+t_0;t_0,u_0)$. This combined with \eqref{d-001} imply that
\begin{align*}
u_t
&\geq\Delta u -\chi\nabla v\cdot\nabla u +u(a_{\inf}-(u^\infty+\epsilon)b_{\sup})
\end{align*}
for $t\ge T^\epsilon$. By comparison principle for parabolic equations again, we have
$$
u(x,t+t_0;t_0,u_0)\ge e^{(a_{\inf}-(u^\infty+\epsilon)b_{\sup})(t-T^\epsilon)}\inf_{x\in\R^N} u(x,T^\epsilon+t_0;t_0,u_0)\quad \forall\,\, t\ge T^\epsilon.
$$
By the boundedness of $u(x,t+t_0,t_0,u_0)$ for $t\ge 0$, we must have
$$
a_{\inf}-(u^\infty+\epsilon)b_{\sup}\le 0\quad \forall\,\, \epsilon>0.
$$
The first inequality in \eqref{asymptotic-lower-bound} then follows.

 Now, if $\liminf_{t\to\infty}\inf_{x\in\R^N}u(x,,t+t_0;t_0,u_0)=0$,
then the second inequality in \eqref{asymptotic-lower-bound} holds trivially. Assume $u_\infty:=\liminf_{t\to\infty}\inf_{x\in\R^N}u(x,,t+t_0;t_0,u_0)>0$.
Then for any $0<\epsilon<u_\infty$, there is $T_\epsilon>0$ such that
$$
u(x,t+t_0;t_0,u_0)\ge  u_\infty-\epsilon\,\,\, {\rm and}\,\,\, v(x,t+t_0;t_0,u_0)\ge \frac{\mu}{\lambda}(u_\infty-\epsilon)\quad \forall \,\, t\ge T_\epsilon.
$$
 This combined with \eqref{global-eq01} yields that
\begin{align*}
u_t
&\le \Delta u -\chi\nabla v\cdot\nabla u +u(a_{\sup}-(u_\infty-\epsilon)b_{\inf})
\end{align*}
for $t\ge T_\epsilon$. By comparison principle for parabolic equations, we have
$$
u(x,t+t_0;t_0,u_0)\le e^{(a_{\sup}-(u_\infty-\epsilon)b_{\inf})(t-T_\epsilon)}\|u(\cdot,T_\epsilon+t_0;t_0,u_0)\|\quad \forall \,\, t\ge T_\epsilon.
$$
This together with the first inequality in \eqref{asymptotic-lower-bound} implies that
$$
a_{\sup}-(u_\infty-\epsilon)b_{\inf}\ge 0\quad \forall\,\, 0<\epsilon<u_\infty.
$$
The second inequality in \eqref{asymptotic-lower-bound} then follows.

(iii)
Let $x\in\R^N$ and  $t\ge 0$ be fixed. Define
$$
f(y)= \lambda v(x+y, t+t_0;t_0,u_0)- \mu u(x+y, t+t_0;t_0,u_0), \quad \forall y\in B(0,3)
$$
and
$$
\phi(y)= v(x+y, t+t_0;t_0,u_0), \quad \forall y\in \bar{B}(0,3)
$$
Let $G_1$ be the solution of
$$
\begin{cases}
\Delta G_{1}=f, \quad y\in B(0,3)\cr
G_{1}=0, \quad \text{on } \ \partial B(0,3).
\end{cases}
$$
Choose $p\gg N$ such that $W^{2,p}(B(0,3))\subset C^{1+\nu}_{\rm unif}(B(0,3))$ (with continuous embedding). Thus, by regularity for elliptic equations, there is $c_{1,\nu}>0$ (depending only on $\nu$, $N$ and the Lebesgue measure $|B(0,3)|$ of $B(0,3)$) such that
\begin{equation}\label{z3}
\|G_{1}\|_{C^{1+\nu}_{\rm unif}(B(0,3))}\leq c_{1,\nu}\|f\|_{L^{p}(B(0,3))}.
\end{equation}

Next, define
$$G_2(y)=v(x+y, t+t_0;t_0,u_0)-G_{1}(y), \quad \forall y\in\bar{B}(0,3).$$ Hence $G_2$ solves
$$
\begin{cases}
\Delta G_{2}=0, \quad y\in B(0,3)\cr
G_{2}(y)=\phi(y), \quad y\in \partial B(0,3).
\end{cases}
$$
Thus, (see \cite[page 41]{Evans}),
$$
G_{2}(y)=\frac{2-\|y\|^{2}}{2N\omega_{N}}\int_{\partial B(0,3)}\frac{\phi(z)}{|y-z|^{N}}dS(z), \quad \forall y\in B(0,3),
$$
where $\omega_{N}=|B(0,1)|$ is the Lebesgue measure of $B(0,1)$, and
\begin{equation}\label{z4}
\partial_{y_i}G_2(y)=-\frac{y_{i}}{N\omega_{N}}\int_{\partial B(0,3)}\frac{\phi(z)}{|y-z|^{N}}dS(z)+ \frac{2-\|y\|^{2}}{2\omega_{N}}\int_{\partial B(0,3)}\frac{(y_i-z_i)\phi(z)}{|y-z|^{N+2}}dS(z), \forall y\in B(0,3).
\end{equation}
But
$$|y+h-z|\geq |z|-|y+h|\geq 1,\quad \forall z\in\partial B(0,3),\ y,h\in {B}(0,1)$$
and
$$
||y+h-z|-|y-z||\leq |h|, \forall y,h,z\in\R^N.
$$
It follows from \eqref{z4}, that there is $c_{2,\nu}>0$ (depending only on $\nu$, $N$ and $|B(0,3)|$) such that
$$
|\partial_{y_i}G_2(y+h)-\partial_{y_i}G_2(y) |\leq c_{2,\nu}|h|^{\nu}\|\phi\|_{\infty},\quad  \forall y,h\in {B}(0,1).
$$
Combining the last inequality with \eqref{z3}, there is $c_{\nu}(N,P)$(depending only on $\nu$, $N$ and $|B(0,3)|$) such that
\begin{equation}\label{global-eq07}
\|G_{1}+G_2\|_{C^{1+\nu}_{\rm unif}(B(0,1))}\leq c_{\nu}[\|f\|_{\infty}+\|\phi\|_{\infty}].
 \end{equation}
Note that $v(x+h, t+t_0;t_0,u_0)=(G_{1}+G_{2})(h)$,  thus (iii) follows from \eqref{u-upper-bound-eq2}, \eqref{v-bound-eq1}, \eqref{v-bound-eq2},  and \eqref{global-eq07}.
 \end{proof}

\section{Pointwise persistence}

In this section we explore the pointwise persistence of positive
classical solutions and prove Theorem \ref{Main-thm1}. In order to do so, we first prove some lemmas. The next Lemma provides a finite time pointwise persistence for solutions $(u(x,t;t_0,u_0),v(x,t;t_0,u_0))$ of \eqref{P} with strictly positive function $u_0$.

\begin{lem} \label{main-lem1}
 Suppose that {\bf (H1)} holds. Then for every $T>0$, $t_0\in\R,$ and  for every nonnegative initial function $ u_0\in C^{b}_{\rm unif}(\R^N)$, there holds that
\begin{equation}\label{eq1-main-lem1}
\inf_{x\in\R^N}u(x,t+t_0;t_0,u_0)\geq u_{0\inf}e^{t(a_{\inf}-b_{sup}\|u_0\|_{\infty}e^{T a_{\sup}})}, \quad \forall 0\leq t\leq T.
\end{equation}
In particular for every $T>0$ and  for every nonnegative initial $u_0\in C^b_{\rm unif}(\R^N)$ satisfying $\|u_0\|_{\infty}\leq M_{T}:=\frac{a_{\inf}e^{-a_{\sup}T}}{b_{\sup}}$, we have that
\begin{equation}\label{eq2-main-lem1}
\inf_{x\in\R^N}u(x,t+t_0;t_0,u_0)\geq \inf_{x\in\R^N}u_{0}(x), \quad \forall 0\leq t\leq T,\ \forall t_{0}\in\R.
\end{equation}
\end{lem}

\begin{proof} Let $t_0\in\R$ and  $u_0\in C^b_{\rm unif}(\R^N)$, $u_0\geq 0$, be given. Since {\bf (H1)} holds, it follows from Theorem \ref{global-existence-tm} that  $(u(\cdot,t+t_0;t_0,u_0), u(\cdot,t+t_0;t_0,u_0))$ is defined for all $t\geq 0.$  By  \eqref{u-upper-bound-eq1} and \eqref{v-bound-eq1},
$$
\chi\lambda \|v(\cdot,t+t_0;t_0;u_0)\|_{\infty}\leq \chi\mu\|u(\cdot,t+t_0;t_0;u_0)\|_{\infty}\leq\chi\mu \|u_0\|_{\infty}e^{a_{\sup}t},\quad \forall\ t\geq 0.
$$
 Hence, for every $t_0< t\leq t_0+T$, it follows from the previous inequality and \eqref{d-001}  that
\begin{align}\label{aa-eq1}
u_t
&\geq \Delta u -\chi\nabla v\cdot\nabla u +u(a_{\inf}-b_{\sup}\|u_0\|_{\infty}e^{a_{\sup}T}).
\end{align}
Thus, by comparison principle for parabolic equations, it follows from \eqref{aa-eq1} that
\begin{equation}\label{aa-eq2}
\inf_{x\in\R^N}u(x,t+t_0;t_0,u_0)\geq u_{0\inf}e^{t(a_{\inf}-b_{sup}\|u_0\|_{\infty}e^{a_{\sup}T})}, \quad \forall\, 0\leq t\leq T,\,\, T>0,\,\,  t_0\in\R.
\end{equation}
Observe that $\|u_0\|_{\infty}\leq M_{T}:=\frac{a_{\inf} e^{-a_{\sup}T}}{b_{\sup}}$ implies that $a_{\inf}-b_{\sup}\|u_0\|_{\infty}e^{a_{\sup}T}\geq 0$. This combined with \eqref{aa-eq2} yields \eqref{eq2-main-lem1}.
\end{proof}

\begin{rk} We note that a slight modification of the proof of Lemma \ref{main-lem1} yields that if {\bf (H1)} does not hold then \begin{equation}\label{eq1-main-lem1'}
\inf_{x\in\R^N}u(x,t+t_0;t_0,u_0)\geq u_{0\inf}e^{t(a_{\inf}-(b_{\sup}+\chi\mu)\|u_0\|_{\infty}e^{a_{\sup}T})}, \quad \forall\,\,  0\leq t\leq T<T_{\max}(u_0),
\end{equation}
for every nonnegative initial $u_0\in C^{b}_{\rm unif}(\R^N)$.
Hence for every initial function $ u_0\in C^{b}_{\rm unif}(\R^N)$ with $\inf_{x\in\R^N}u_0(x)>0$, it always holds that
$$
\inf_{x\in\R^N, 0\le t\le T}u(x,t+t_0;t_0,u_0)>0,\quad \forall\ 0\leq T<T_{\max}(u_0), \forall t_0\in\R.
$$
It should be noted \eqref{eq1-main-lem1'} and \eqref{eq1-main-lem1} do not implies the  pointwise persistence of $ u(x,t+t_0;t_0,u_0)$.
\end{rk}

\begin{lem}\label{continuity with respect to open compact topology}
Assume that (H1) holds. Let $u_0\in C^{b}_{\rm unif}(\R^N)$, $\{u_{0n}\}_{n\geq 1}$ be a sequence of nonnegative functions in $C^{b}_{\rm unif}(\R^N)$, and let $\{t_{0n}\}_{n\geq 1}$ be a sequence of real numbers. Suppose that $0\leq u_{0n}(x)\leq M:=\frac{a_{\sup}}{b_{\inf}-\chi\mu}$ and $\{u_{0n}\}_{n\geq 1}$ converges uniformly locally to $u_{0}$. Then there exist a subsequence {$\{t_{0n'}\}$ of $\{t_{0n}\}$}, functions $a^*(x,t), b^*(x,t)$  such that $(a(x,t+t_{0n'}),b(x,t+t_{0n'}))\to (a^*(x,t),b^*(x,t))$ locally uniformly as $n'\to\infty$, and  $u(x,t+t_{0n'};t_{0n'},u_{0n'})\to u^*(x,t;0,u_0)$ locally uniformly in $C^{2,1}(\R^N\times(0, \infty))$ as $n'\to\infty$, where $(u^*(x,t;0,u_0),v^*(x,t;0,u_0)$ is the classical solution of
\begin{equation*}
\begin{cases}
u_{t}(x,t)=\Delta u(x,t)-\chi\nabla\cdot (u(x,t)\nabla v(x,t))+(a^{*}(x,t)-b^{*}(x,t)u(x,t))u(x,t), \quad x\in\R^N\cr
0=(\Delta-\lambda I)v^*(x,t)+\mu u^*(x,t), \quad x\in\R^N\cr
u^*(x,0)=u_{0}(x), \quad x\in\R^N.
\end{cases}
\end{equation*}
\end{lem}

\begin{proof} Without loss of generality, by the Arzela Ascoli's Theorem, we may suppose that $(a(x,t+t_{0n}),b(x,t+t_{0n}))\to (a^*(x,t),b^*(x,t))$ locally uniformly in $\R^N\times\R$ as $n\to \infty$.
Recall that $(u(x,t+t_{0n};t_{0n},u_{0n}),v(x,t+t_{0n};t_{0n},u_{0n}))$ satisfies for $x\in\R^N,\ t>0$,
\begin{align*}
u_t(x,t+t_{0n};t_{0n},u_{0n})&=\Delta u(x,t+t_{0n};t_{0n},u_{0n}) -\chi\nabla\cdot (u(x,t+t_{0n};t_{0n},u_{0n}) \nabla v(x,t+t_{0n};t_{0n},u_{0n}))\nonumber\\
& + (a(x,t+t_{0n})-b(x,t+t_{0n})u(x,t+t_{0n};t_{0n},u_{0n}))u(x,t+t_{0n};t_{0n},u_{0n}).
\end{align*}
So, by variation of constant formula, we have that
\begin{align}\label{cont-eq1}
u(\cdot,t+t_{0n};t_{0n},u_{0n})= & \underbrace{e^{t(\Delta-I)}u_{0n}}_{I_{0n}}(t)+ \underbrace{\int_0^te^{(t-s)(\Delta-I)}(a(\cdot,s+t_{0n})+1)u(\cdot,s+t_{0n};t_{0n},u_{0n})ds}_{I_{1n}(t)}\nonumber\\
& -\chi\underbrace{\int_{0}^{t}e^{(t-s)(\Delta-I)}\nabla\cdot(u(\cdot,s+t_{0n};t_{0n},u_{0n})\nabla v(\cdot,t+t_{0n};t_{0n},u_{0n}))ds}_{I_{2n}(t)}\nonumber\\
& -\underbrace{\int_0^t e^{(t-s)(\Delta-I)}b(\cdot,s+t_{0n})u^2(\cdot,s+t_{0n};t_{0n},u_{0n})ds}_{I_{3n}(t)}, \forall \ t>0,
\end{align}
where $\{e^{t(\Delta-I)}\}_{t\geq 0}$ denotes the analytic semigroup generated on $X^0:=C^{b}_{\rm unif}(\R^N)$ by $\Delta-I$. Let $X^{\beta}$, $\beta >0$, denote the fractional power spaces associated with $ \Delta-I$. Let $0<\beta<\frac{1}{2}$ be fixed.


There is a constant $C_{\beta}>0$, (see \cite{Dan Henry}), such that
$$
\|I_{0n}(t+h)-I_{0n}(t)\|_{X^\beta}=\leq C_{\beta}h^{\beta}t^{-\beta}\|u_{0n}\|_{\infty}\leq C_{\beta}h^{\beta}t^{-\beta}M,
$$
\begin{align*}
\|I_{1n}(t+h)-I_{1n}(t)\|_{X^{\beta}} &\leq C_{\beta}(a_{\sup}+1)M\left[h^{\beta}\int_{0}^{t}\frac{e^{-(t-s)}}{(t-s)^{\beta}}ds+\int_{t}^{t+h}\frac{e^{-(t+h-s)}}{(t+h-s)^{\beta}}ds\right]\nonumber\\
& \leq C_{\beta}(a_{\sup}+1)M\left[ h^{\beta}\Gamma(1-\beta)+\frac{h^{1-\beta}}{1-\beta}\right],
\end{align*}
and
$$
\|I_{2n}(t+h)-I_{2n}(t)\|_{X^{\beta}} \leq C_{\beta}b_{\sup}M^2\left[ h^{\beta}\Gamma(1-\beta)+\frac{h^{1-\beta}}{1-\beta}\right].
$$
It follows from \cite[Lemma 3.2]{SaSh1}  that
\begin{align*}
\|I_{3n}(t+h)-I_{3n}(t)\|_{X^{\beta}} &\leq \frac{\mu\sqrt{N}C_{\beta}M^2}{\sqrt{\lambda}}\left[h^{\beta}\int_{0}^{t}\frac{e^{-(t-s)}}{(t-s)^{\beta+\frac{1}{2}}}ds+\int_{t}^{t+h}\frac{e^{-(t+h-s)}}{(t+h-s)^{\beta+\frac{1}{2}}}ds\right]\nonumber\\
& \leq \frac{\mu\sqrt{N}C_{\beta}M^2}{\sqrt{\lambda}}\left[ h^{\beta}\Gamma(1-\beta)+\frac{h^{1-\beta}}{1-\beta}\right].
\end{align*}
Hence the function $(0,\infty)\ni t\mapsto u(\cdot,t+t_{0n};t_{0n},u_{0n})\in X^{\beta}$ is locally uniformly H\"older's continuous.  It then follows from Arzela-Ascoli Theorem and \cite[Theorem 15]{Friedman} that there is a subsequence $\{t_{0n'}\}$  of
$\{t_{0n}\}$ and a function $u\in C^{2,1}(\R^N\times (0,\infty))$ such that $u(x,t+t_{0n'};t_{0n'},u_{0n'})$ converges to $u(x,t)$ locally uniformly in $C^{2,1}(\R^N\times (0, \infty))$ as $n'\to \infty$. Furthermore, taking $v(x,t)=\mu(\lambda I-\Delta )^{-1}u(x,t)$, we have that
\begin{align*}
u_{t}(x,t)=\Delta u(x,t)-\chi\nabla\cdot (u(x,t)\nabla v(x,t))+(a^{*}(x,t)-b^{*}(x,t)u(x,t))u(x,t), \quad x\in\R^N
\end{align*}
for $t>0$.
Since $u_{0n'}(x)\to u_0(x)$ locally uniformly as $n\to \infty$, it is not hard to show from \eqref{cont-eq1} that $u(x,t)$ satisfies
\begin{align}\label{int-eq}
u(x,t)=& e^{t(\Delta-I)}u_0-\chi\int_{0}^{t}e^{(t-s)(\Delta-I)}\nabla\cdot (u(\cdot,s)\nabla v(\cdot,s))ds\nonumber\\
&+\int_0^t e^{(t-s)(\Delta-I)}((1+a^*(\cdot,s))u-b^*(\cdot,s)u^2(\cdot,s)))ds.
\end{align}
Note that $(u^*(x,t;0,u_0),v^*(x,t;0,u_0))$ also satisfies the integral equation \eqref{int-eq}. It thus follows from the Generalized  Grownwall's inequality, \cite[Lemma 2.5]{SaSh1} that $u(x,t)=u^*(x,t;0,u_0)$.
\end{proof}

\begin{lem}
\label{main-new-lm1}
Assume that (H1) holds. For every $M>0$, $\varepsilon>0$, and $T>0$,   there exist $L_0=L(M,T,\varepsilon)\gg 1$ and $\delta_0=\delta_0(M,\varepsilon)$ such that for every initial function $u_0\in C^{b}_{\rm unif}(\R^N)$ with $0\leq u_0\leq M$ and   every $L\geq L_0$,
\begin{equation}\label{Local-uniform-boun-for-u}
  u(x,t+t_0;t_0,u_0)\leq \varepsilon,\quad \forall \ 0\leq t\leq T,\ t_0\in\R, \ \forall \,\, |x|_{\infty}<2L
\end{equation}
whenever $ 0\leq u_0(x)\leq \delta_0$ for   $|x|_{\infty}<3L$.
\end{lem}

\begin{proof}
 It follows from \eqref{global-eq01} and comparison principle for parabolic equations that
 \begin{equation}\label{global-eq08}
 0\leq u(x,t+t_0;t_0,u_0)\leq U(x,t+t_0;t_0,u_0),\quad \forall\ x\in\mathbb{R}^{N},\,\, t\geq 0,
 \end{equation}
 where $U$ solves
 \begin{equation}\label{global-eq09}
\begin{cases} U_{t}=\Delta U -\chi\nabla v(\cdot,\cdot;t_0,u_0)\cdot\nabla U +a_{\sup}U, \quad  t>t_0\cr
U(\cdot,t_0)=u_0
 \end{cases}
 \end{equation}
  It follows from Theorem \ref{global-existence-tm} (iii) and  \cite[Theorem 12]{Friedman}  that $U((x,t_0+t;t_0,u_0))$ can be written in the form
  \begin{align}\label{d00}
  U(x,t_0+t;t_0,u_0)&=\int_{\R^N}\Gamma(x,t,y,0)u_0(y)dy .
  \end{align}
Moreover, for every $0<\lambda_0<1$,  there is a constant $K_2=K_{2}(\lambda_0,N,\nu,K_1, T)$, where $K_1$ is given by  Theorem \ref{global-existence-tm} (iii), such that
\begin{equation}\label{d01}
\left|\Gamma(x,t,y,\tau)\right|\leq K_2\frac{e^{-\frac{\lambda_0|x-y|^2}{4(t-\tau)}}}{(t-\tau)^{\frac{N}{2}}} \ \ \text{and}\ \ \left|\partial_{x_i}\Gamma(x,t,y,\tau)\right|\leq K_2\frac{e^{-\frac{\lambda_0|x-y|^2}{4(t-\tau)}}}{(t-\tau)^{\frac{N+1}{2}}}, \quad \forall\, x\in\R^N,\, \forall\, \tau\leq t\leq \tau + T.
\end{equation}
We then have
\begin{align*}
 U(x,t_0+t;t_0,u_0)&\le K_2\int_{\R^N}\frac{e^{-\frac{\lambda_0|x-y|^2}{4 t}}}{t^{\frac{N}{2}}}u_0(y)dy\\
 &=K_2\int_{R^N}e^{-\frac{\lambda_0}{4} |z|^2} u_0(x+t^{\frac{1}{2}}z)dz\\
 &\le  K_2\Big[ \int_{|z|_{\infty}\le \frac{L}{\sqrt{T}}} e^{-\frac{\lambda_0}{4} |z|^2} u_0(x+t^{\frac{1}{2}}z)dz
 +\int_{|z|_{\infty}\ge \frac{L}{\sqrt{T}}}e^{-\frac{\lambda_0}{4} |z|^2} u_0(x+t^{\frac{1}{2}}z)dz\Big].
\end{align*}
This implies that for $|x|\le 2L$,
\begin{align}\label{d02}
U(x,t_0+t;t_0,u_0)&\le K_2 \delta_0\int_{\R^N}e^{-\frac{\lambda_0}{4} |z|^2}  dz+K_2\|u_0\|_{\infty} \int_{|z|_{\infty}\ge \frac{L}{\sqrt{T}}}e^{-\frac{\lambda_0}{4} |z|^2} dz\cr
& \leq K_2\delta_0\left(\frac{4\pi}{\lambda_0}\right)^{\frac{N}{2}}+K_2M\int_{|z|_{\infty}\ge \frac{L}{\sqrt{T}}}e^{-\frac{\lambda_0}{4} |z|^2} dz.
\end{align}
Take $\delta_0=\frac{\varepsilon}{2 K_2}\left(\frac{4\pi}{\lambda_0}\right)^{-\frac{N}{2}}$ and choose $L_0\gg 1$ such that $ \int_{|z|_{\infty}\ge \frac{L_0}{\sqrt{T}}}e^{-\frac{\lambda_0}{4} |z|^2} dz<\frac{\varepsilon}{2K_2M}$, it follows from \eqref{d02} that for every $L\geq L_0$,  there holds that  $U(x,t+t_0;t_0,u_0)\leq \varepsilon$ for every $|x|_{\infty}\leq 2L$ whenever $u_0(x)\leq \delta_0$ for all $|x|_{\infty}\leq 3L$.
This combined with \eqref{global-eq08} yields the lemma.
\end{proof}

\begin{lem}\label{Main-lem2'} Suppose that {\bf (H2)} holds. Consider the sequence $(\underline{M}_{n},\overline{M}_{n})_{n\geq0}$ defined inductively by  $\underline{M}_{0}=0$ and
\begin{equation}\label{M-n def}
\overline{M}_{n}=\frac{a_{\sup}-\chi\mu\underline{M}_{n}}{b_{\inf}-\chi\mu}, \quad \text{and}\quad \underline{M}_{n+1}=\frac{a_{\inf}-\chi\mu\overline{M}_{n}}{b_{\sup}-\chi\mu}, \quad \forall \ n\geq 0.
\end{equation}
Then  for every $n\ge 0$, it holds that $$
\underline{M}_{n+1}>\underline{M}_{n}\geq 0\quad \text{and}\quad  \overline{M}_{n}>\overline{M}_{n+1}> 0.
$$ Moreover, we have that
\begin{equation*}
\lim_{n\to\infty}(\underline{M}_{n},\overline{M}_{n})=(\underline{M},\overline{M}),
\end{equation*}
where $\underline{M}$ and $\overline{M}$ are given by \eqref{attracting-rect-eq2} and \eqref{attracting-rect-eq3}, respectively.
\end{lem}
\begin{proof}
For every $n\geq 0$, it holds that
\begin{equation} \label{M-n-eq01}
\underline{M}_{n+1}=\frac{(b_{\inf}-\chi\mu)a_{\inf}-\chi\mu a_{\sup}+(\chi\mu)^2\underline{M}_{n}}{(b_{\inf}-\chi\mu)(b_{\sup}-\chi\mu)}
\end{equation}
and
\begin{equation}\label{M-n-eq02}
\overline{M}_{n+1}=\frac{(b_{\sup}-\chi\mu)a_{\sup}-\chi\mu a_{\inf}+(\chi\mu)^2\overline{M}_{n}}{(b_{\inf}-\chi\mu)(b_{\sup}-\chi\mu)}.
\end{equation}
Thus, since $\underline{M}_0= 0$, $\overline{M}_0=\frac{a_{\sup}}{b_{\inf}-\chi\mu}>0$, and {\bf (H2)} holds, it follows by mathematical induction that $\underline{M}_n\geq 0$ and $\overline{M}_n\geq 0$ for every $n\geq 0.$  Therefore, it follows from \eqref{M-n def} that
$$0\leq  \underline{M}_n\leq \frac{a_{\inf}}{b_{\sup}-\chi\mu} \quad \text{and}\quad 0\leq \overline{M}_n \leq \frac{a_{\sup}}{b_{\inf}-\chi\mu},\quad \forall\, n\geq 0.$$ Observe that $\underline{M}_{0}<\underline{M}_1$. Hence, \eqref{M-n-eq01} implies that $\underline{M}_{n}<\underline{M}_{n+1}$ for every $n\geq 0$. Similarly, we have that $\overline{M}_0>\overline{M}_{1}$. Hence \eqref{M-n-eq02} implies that $\overline{M}_{n+1}<\overline{M}_{n}$ for every $n\geq 0$. Thus the sequence $(\underline{M}_{n},\overline{M}_n)$ is convergent. By passing to limit in \eqref{M-n-eq01}  and \eqref{M-n-eq02},  it is easily seen that $ \lim_{n\to\infty}(\underline{M}_{n},\overline{M}_n)=(\underline{M},\overline{M})$, where $\underline{M}$ and $\overline{M}$ are given by \eqref{attracting-rect-eq3} and \eqref{attracting-rect-eq2} respectively.
\end{proof}

\begin{lem}
\label{main-new-lm2}
 For fixed $T>0$, there is $0<\delta_0^*(T)<M^+=\frac{a_{\sup}}{b_{\inf}-\chi\mu}+1$ such that for any $0< \delta\le \delta_0^*(T)$
 and for any $u_0$ with
$\delta\le u_0\le M^+$,
\begin{equation}
\label{aux-eq5-1}
\delta\le u(x,t_0+T;t_0,0,u_0)\le M^+\quad \forall\,\, x\in\R^N, \ \forall\ t_0\in\R.
\end{equation}
\end{lem}

\begin{proof}
 We divide the proof into four steps.

 First of all,  let $a_0=\frac{a_{\inf}}{3}$ and
$$D_L=\{x\in\R^N\,|\, |x_i|<L\quad {\rm for}\,\,\, i=1,2,\cdots,N\}.
$$
 Consider
\begin{equation}
\label{aux-eq1-1}
\begin{cases}
u_t=\Delta u+a_0 u,\quad x\in D_L\cr
u=0,\quad x\in\partial  D_L,
\end{cases}
\end{equation}
and its associated eigenvalue problem
\begin{equation}
\label{aux-eq1-2}
\begin{cases}
\Delta u+a_0 u=\sigma u,\quad x\in D_L\cr
u=0,\quad x\in\partial D_L.
\end{cases}
\end{equation}
Let $\sigma_{_{L}}$ be the principal eigenvalue of \eqref{aux-eq1-2} and $\phi_L(x)$ be its principal eigenfunction
with $\phi_L(0)=1$. Note that
$$\phi_L(x)=\Pi_{i=1}^N \cos\big(\frac{\pi}{2L}x_i\big)
\quad {\rm and}\quad 0<\phi_L(x)\le \phi_L(0),\quad \forall x\in D_L.
$$
Note also that $u(x,t)=e^{\sigma_{_{L}} t}\phi_L(x)$ is a solution of \eqref{aux-eq1-1}.
Let $u(x,t;u_0)$ be the solution of \eqref{aux-eq1-1} with $u_0\in C(\bar D_L)$.
Then
\begin{equation}
\label{aux-eq1-3}
u(x,t;\kappa \phi_L)=\kappa e^{\sigma_{_{L}} t}\phi_L(x)
\end{equation}
for all $\kappa\in\R$.

In the following, let $L_0\gg 0$ be such that
$\sigma_{_{L}}>0 \quad \forall L\ge  L_0.$

\medskip

\noindent {\bf Step 1.} Let $T>0$ be fixed. Consider
\begin{equation}
\label{aux-eq2-1}
\begin{cases}
u_t=\Delta u+b_\epsilon(x,t) \cdot \nabla u+a_0  u,\quad x\in D_L\cr
u=0,\quad x\in\partial D_L,
\end{cases}
\end{equation}
where $|b_\epsilon(x,t)|<\epsilon$ for $x\in \bar D_L$ and
$t_0\le t\le t_0+T$.
Let $u_{b_\epsilon,L}(x,t
;t_0,u_0)$ be the solution of \eqref{aux-eq2-1} with $u_{b_\epsilon,L}(x,t_0;t_0,u_0)=u_0(x)$.

We claim that {\it there is $\epsilon_0(T)>0$ such that for any $L\ge L_0$, $\kappa>0$, and $0\le \epsilon\le \epsilon_0(T)$,
\begin{equation}
\label{aux-eq2-2}
u_{b_\epsilon,L}(0,t_0+T;t_0,\kappa \phi_L)\geq e^{\frac{T\sigma_{{L_{0}}}}{2}}\kappa>\kappa
\end{equation}
provided that $|b_\epsilon(x,t)|<\epsilon$ for $x\in D_L$; and for any
$L\ge L_0$ and $0\le \epsilon\le \epsilon_0(T)$,
\begin{equation}
\label{aux-eq2-2-0}
0\le u_{b_\epsilon,L}(x,t+t_0;t_0,\kappa \phi_L)\le e^{a_0 t}\kappa \quad \forall \,\, 0\le t\le T,\,\, x\in D_L.
\end{equation}}

In fact, by \eqref{aux-eq1-3},  there is $\epsilon_0(T)>0$ such that for any $0\le\epsilon\le\epsilon_0(T)$,
\begin{equation}
\label{aux-eq2-2-1}
u_{b_\epsilon,L_0}(0,t_0+T;t_0,\kappa \phi_{L_0})> e^{\frac{T\sigma_{L_0}}{2}}\kappa
\end{equation}
provided that  $|b_\epsilon(x,t)|<\epsilon$ for $x\in D_{L_0}$.
Note that for $L\ge L_0$,
$$
\phi_L(x)\ge \phi_{L_0}(x)\quad \forall\,\, x\in D_{L_0}.
$$
and
$$
\begin{cases}
\partial_t u_{b_\epsilon,L}(\cdot,t_0+t; t_0,\kappa\phi_L)=\Delta u_{b_\epsilon,L}(x,t_0+t; t_0,\kappa\phi_L)+b_\epsilon(x,t) \cdot \nabla u_{b_\epsilon,L}(x,t_0+t; t_0,\kappa\phi_L)\cr
\qquad \qquad \qquad \qquad \qquad \qquad+a_0  u_{b_\epsilon,L}(x,t_0+t; t_0,\kappa\phi_L),\quad x\in D_{L_0},\cr
u_{b_\epsilon,L}(x,t_0+t; t_0,\kappa\phi_L)>0,\quad x\in\partial D_{L_0}.
\end{cases}
$$
Then by comparison principle for parabolic equations,
\begin{align*}
u_{b_\epsilon,L}(x,t_0+t; t_0,\kappa\phi_L)\ge u_{b_\epsilon,L_0}(x,t_0+t;t_0,\kappa\phi_{L_0})
\end{align*}
for $x\in D_{L_0}$,
which together with  \eqref{aux-eq2-2-1} implies \eqref{aux-eq2-2}.
\eqref{aux-eq2-2-0} follows directly from  comparison principle for parabolic equations.

\medskip

\noindent {\bf Step 2.} Consider
\begin{equation}
\label{aux-eq3-1}
\begin{cases}
u_t=\Delta u+b_\epsilon(x,t)\cdot \nabla u+ u(2 a_0-c(x,t) u),\quad x\in D_L\cr
u=0,\quad x\in \partial D_L,
\end{cases}
\end{equation}
where $0\le c(x,t)\le b_{\sup}$. Let $u(x,t;t_0,u_0)$ be the solution of \eqref{aux-eq3-1} with
$u_{\varepsilon}(x,t_0;t_0,u_0)=u_0(x)$.
Assume $L\ge L_0$ and $0\le\epsilon\le\epsilon_0(T)$.

 We claim that {\it
\begin{equation}
\label{aux-eq3-2}
u_{\varepsilon}(0,t_0+T;t_0,\kappa\phi_L)\ge  e^{\frac{T\sigma_{L_0}}{2}}\kappa
\end{equation}
provided that $0<\kappa \le \kappa_0(T):=\frac{a_0e^{-a_0T}}{ b_{\sup}}$.}

\smallskip

Note that \eqref{aux-eq2-2-0} yields,
\begin{align*}
&\partial_t u_{b_\epsilon,L}(x,t;t_0,\kappa \phi_L)-\Delta u_{b_\epsilon,L}(x,t;t_0,\kappa \phi_L)-b_{\epsilon}(x,t)\cdot\nabla u_{b_\epsilon,L}(x,t;t_0,\kappa\phi_L)\\
& -u_{b_\epsilon,L}(x,t;t_0,\kappa\phi_L)\big(2 a_0-c(x,t) u_{b_\epsilon,L}(x,t;t_0,\kappa \phi_L)\big)\\
&=-u_{b_\epsilon,L}(x,t;t_0,\kappa\phi_L)\big(a_0-c(x,t)u_{b_\epsilon,L}(x,t;t_0,\kappa \phi_L)\big)\\
&\le 0\quad {\rm for}\quad t_0\le t\le t_0+T,\quad x\in D_L
\end{align*}
when $0<\kappa \le \frac{a_0e^{-a_0T}}{ b_{\sup}}$. Then by comparison principal for parabolic equations,
$$
u_{\varepsilon}(x,t;t_0,\kappa \phi_L)\ge u_{b_\epsilon,L}(x,t;t_0,\kappa \phi_L)\quad {\rm for}\quad t_0\le t\le t_0+T,\quad x\in D_L.
$$
This together with \eqref{aux-eq2-2} implies \eqref{aux-eq3-2}.

\medskip

\noindent {\bf Step 3.} For any given $x_0\in\R^N$, consider
\begin{equation}
\label{aux-eq4-1}
u_t=\Delta u-\chi \nabla v\cdot \nabla u+u(a(x+x_0,t)-\chi\lambda v(x,t;t_0,x_0,u_0)-(b(x+x_0,t)-\chi\mu) u),\quad x\in\R^N,
\end{equation}
where $v(x,t;t_0,x_0,u_0)$ is the solution of
$$
0=\Delta v-\lambda v+\mu u,\quad x\in\R^N.
$$
Let $u(x,t;t_0,x_0,u_0)$ be the solution of \eqref{aux-eq4-1} with $u(x,t_0;t_0,x_0,u_0)=u_0(x)$.
Let $\epsilon_0(T)>0$ and $\kappa_0(T)>0$ be as in Steps 1 and 2, respectively.

We claim that
{\it  there is $0<\delta_0(T)\le \kappa_0(T)$ such that for any $u_0\in C_{\rm unif}^b(\R^N)$ with
$0\leq u_0\leq M^+$ and $u_0(x)<\delta_0(T)$ for $|x_i|\le 3L$, $i=1,2,\cdots,N$,
\begin{equation}
\label{aux-eq4-2}
0\le \lambda  v(x,t;t_0,x_0,u_0)\le \frac{a_0}{2\chi},\ \ | \nabla v(x,t;t_0,x_0,u_0)|<\frac{\epsilon_0(T)}{2\chi} \ \ {\rm for}\ \ t_0\le t\le t_0+T,\ x\in D_L,\ x_0\in\R^N
\end{equation}
provided that $L\gg 1$}.

Indeed, let $0<\varepsilon\leq \varepsilon_0(T)$ be fixed.  Lemma \ref{main-new-lm1} implies that there is $\delta_{1}=\delta_1(M^+,\varepsilon)$ and $L_1=L_1(M^+,T,\varepsilon)>L_0$ such that for every $L\geq L_1$, there holds
\begin{equation}\label{aux-eq4-2-0}
u(x,t+t_0;t_0,x_0,u_0)\leq \varepsilon,\quad \forall \ 0\leq t\leq T,\ t_0\in\R, \ \forall x_0\in\R^N,\ \forall |x|\in D_{2L}
\end{equation}
whenever $0\leq u_0(x)\leq \delta_1, \ \forall\  x\in D_{3L}$.

Next, note that
 \begin{equation*}
v(x,t_0+t;t_0,x_0,u_0)=\mu \int_{0}^{\infty}\int_{\R^N}\frac{e^{-\lambda s}}{(4\pi s)^{\frac{N}{2}}}e^{-\frac{|x-z|^{2}}{4s}}u(z,t;t_0,x_0,u_0)dzds
\end{equation*}
and \begin{equation*}
\partial_{x_i}v(x,t_0+t;t_0,x_0,u_0)=\mu \int_{0}^{\infty}\int_{\R^N}\frac{(z_i-x_i)e^{-\lambda s}}{2s(4\pi s)^{\frac{N}{2}}}e^{-\frac{|x-z|^{2}}{4s}}u(z,t;t_0,x_0,u_0)dzds.
\end{equation*}
Hence, by \eqref{aux-eq4-2-0}, for $L\geq L_1$, $0\leq t\leq T$, and  $|x|_{\infty}< L$,  we have
\begin{align}\label{aux-eq4-2-00}
v(x,t_0+t;t_0,x_0,u_0)& \leq\frac{ \mu}{\pi^{\frac{N}{2}}} \left[\int_{0}^{L}\int_{|z|_{\infty}\leq \frac{L}{2\sqrt{T}}}e^{-\lambda s}e^{-|z|^2}dzds\right]\sup_{0\leq t\leq T, |z|_{\infty}\leq 2L}u(z,t+t_0;t_0,x_0,u_0) \cr
& + \frac{\mu a_{\sup}}{(b_{\inf}-\chi\mu)\pi^{\frac{N}{2}}} \int\int_{s\geq L\, \text{or}\, |z|_{\infty}\geq \frac{L}{2\sqrt{T}}}e^{-\lambda s}e^{-|z|^2}dzds\cr
& \leq \frac{\mu}{\lambda }\varepsilon + \frac{\mu a_{\sup}}{(b_{\inf}-\chi\mu)\pi^{\frac{N}{2}}} \int\int_{s\geq L\, \text{or}\, |z|_{\infty}\geq \frac{L}{2\sqrt{T}}}e^{-\lambda s}e^{-|z|^2}dzds
\end{align}
and
\begin{align}
\label{aux-eq4-2-001}
|\partial_{x_i} v(x,t_0+t;t_0,0,u_0)
|& \leq\frac{\varepsilon\mu}{\pi^{\frac{N}{2}}} \int_{0}^{L}\int_{|z|_{\infty}\leq \frac{L}{2\sqrt{T}}}\frac{|z_i|e^{-\lambda s}e^{-|z|^2}}{\sqrt{s}}dzds \cr
& + \frac{\mu a_{\sup}}{(b_{\inf}-\chi\mu)\pi^{\frac{N}{2}}} \int\int_{s\geq L\, \text{or}\, |z|_{\infty}\geq \frac{L}{2\sqrt{T}}}\frac{|z_i|e^{-\lambda s}e^{-|z|^2}}{\sqrt{s}}dzds\cr
& \leq \frac{ \varepsilon\mu}{\sqrt{\lambda} }+ \frac{\mu a_{\sup}}{(b_{\inf}-\chi\mu)\pi^{\frac{N}{2}}} \int\int_{s\geq L\, \text{or}\, |z|_{\infty}\geq \frac{L}{2\sqrt{T}}}\frac{|z_i|e^{-\lambda s}e^{-|z|^2}}{\sqrt{s}}dzds,\cr
\end{align}
whenever $ 0\leq u_0(x)\leq \delta_1$ for every $|x|_{\infty}\leq 3L$. These together with \eqref{aux-eq4-2-0} implies \eqref{aux-eq4-2}.

Note that
$$
\begin{cases}
u_t\ge \Delta u-\chi\nabla v\cdot\nabla u+u(2 a_0-(b(x+x_0,t)-\chi\mu)u),\quad x\in D_L\cr
u(x,t;t_0,x_0,u_0)>0,\quad x\in \partial D_L.
\end{cases}
$$
Let $\kappa=\inf_{x\in D_L}u_0(x)$.  Then $\kappa\le \delta_0(T) \le\kappa_0(T)$. By comparison principle for parabolic equations,
$$
u(x,t;t_0,x_0,u_0)\ge u_\epsilon (t,x;t_0,\kappa \phi_L)\quad {\rm for}\quad x\in D_L,\quad t_0\le t\le t_0+T.
$$
This together with  the conclusion in Step 2
\begin{equation}\label{n-n0}
u(0,T+t_0;t_0,x_0,u_0)\ge  e^{\frac{T\sigma_{L_0}}{2}}\kappa =e^{\frac{T\sigma_{L_0}}{2}}\inf_{x\in D_L} u_0(x).
\end{equation}

\medskip

\noindent {\bf Step 4.} In this step we claim that {\it there is $0< \delta_0^*(T)< \min\{\delta_0(T),M^+\}$, where $M^+=\frac{a_{\sup}}{b_{\inf}-\chi\mu}$, such that for any $0<\delta\le  \delta_0^*(T)$
 and for any $u_0$ with
$\delta\le u_0\le M^+$,
\begin{equation}
\label{aux-eq5-1}
\delta\le u(x,t_0+T;t_0,0,u_0)\le M^+\quad \forall\,\, x\in\R^N.
\end{equation}}

Assume that the claim does not hold. Then there are $\delta_n\to 0$, $t_{0n}\in\R$, $u_{0n}$ with
$\delta_n\le u_{0n}\le M^+$, and $x_n\in\R^N$ such that
\begin{equation}
\label{aux-eq5-2}
u(x_n,t_{0n}+T;t_{0n},0,u_{0n})<\delta_n.
\end{equation}
Note that
$$
u(x+x_n,t;t_{0n},0,u_{0n})=u(x,t;t_{0n},x_n,u_{0n}(\cdot+x_n)).
$$
Let  $\epsilon_0:=\epsilon(T)>0$, $\delta_0:=\delta_0(T)>0$, and $\kappa_0:=\kappa_0(T)>0$ be fixed and be such that the conclusions in  Steps 2-3 hold.
Let
$$
D_{0n}=\{x\in\R^N\,|\, |x_i|< 3L,\,\, u_{0n}(x+x_n)> \frac{\delta_0}{2}\}.
$$
Without loss of generality, we may assume that $\lim_{n\to\infty} |D_{0n}|$ exists.

\medskip

\noindent {\bf Case 1.} $\lim_{n\to\infty} |D_{0n}|=0.$ We claim that {\it in this case, $|\chi\nabla v(x+x_n,t+t_{0n};t_{0n},0,u_{0n})|<\epsilon_0$ and $0\leq v(x+x_n,t+t_{0n};t_{0n},0,u_{0n})\leq a_0$ for $|x_i|\le L$, $i=1,2,\cdots,N$,  $L\gg 1$ and $n\gg 1$.}
{

Indeed, let $\{\tilde{u}_{0n}\}_{n\geq 1}$ be sequence of elements of   $ C^{b}_{\rm unif}(\R^N)$ satisfying
$$
\begin{cases}
\delta_n\leq \tilde{u}_{0n}(x)\leq \frac{\delta_{0}}{2}, \quad x\in  D_{3L}\ \  \text{and}\cr
\|\tilde{u}_{0n}(\cdot)-u_{0n}(\cdot+x_n)\|_{L^{p}(\R^N)}\to 0, \quad \forall p>1.
\end{cases}
$$
Let $w_{n}(x,t):=u(t+t_{0n},x;t_{0n},x_{n},u_{0n}(\cdot+x_{n}))-u(t+t_{0n},x;t_{0n},x_{n},\tilde{u}_{0n})$ and $v_{n}(x,t):=v(t+t_{0n},x;t_{0n},x_{n},u_{0n}(\cdot+x_{n}))-v(t+t_{0n},x;t_{0n},x_{n},\tilde{u}_{0n})$. Hence $\{(w_{n},v_{n})\}_{n\geq}$ satisfies
\begin{equation}\label{aux-eq5-3}
\begin{cases}
\partial_{t}w_{n}=\Delta w_{n} +b_{n}(t,x)\cdot\nabla w_{n} + f_{n}(t,x)w_n +g_{n}(t,x)v_{n}+h_{n}\cdot\nabla v_{n}, \quad x\in\R^N, t>0\cr
0=\Delta v_n-\lambda v_{n}+\mu w_{n}, \quad  \quad x\in\R^N, t>0\cr
w_{n}(0,x)=u_{0n}(x+x_n)-\tilde{u}_{0n}(x)  , \quad x\in\R^N,
\end{cases}
\end{equation}
where $ b_{n}(t,x)=-\chi\nabla v(t+t_{0n},x+x_n;t_{0n},x_n, u_{0n}(\cdot+x_{n}))$,
$ g_{n}(t,x):=-\chi\lambda u(t+t_{0n},x+x_n;t_{0n},x_{n},\tilde u_{0n})$,  $h_{n}(t,x):=-\chi\nabla u(t+t_{0n},x+x_n;t_{0n},x_{n}, \tilde u_{0n})$, and
\begin{align*}
f_{n}( t-t_{0n},x-x_{n}):= & a(t,x)-\chi\lambda v(t,x;t_{0},x_n,u_{0n}(\cdot+x_n))\nonumber \cr
& -(b(t,x)-\chi\mu)(u(t,x;t_{0},x_n,u_{0n}(\cdot+x_n))+u(t,x;t_{0},x_n,\tilde u_{0n})).
\end{align*}
For $w_{0n}(0,\cdot)\in L^p(\R^N)$, \eqref{aux-eq5-3} has a unique solution $w(t,x;w_{0n})$
with $w(0,x;w_{0n})=w_{0n}(x)$ in $L^p(\R^N)$. Note that $\nabla \cdot (w_{n} b_{n})= b_{n}\cdot\nabla w_{n}+w_{n}\nabla\cdot b_{n}$ and $\nabla \cdot b_{n}=-\chi (\lambda v -\mu u)(t+t_{0n},x;t_{0n},x_{n},u_{0n}(\cdot+x_{n}))$.  Hence
$$
\partial_{t}w_{n}=\Delta w_{n}+\nabla\cdot(w_n b_n)  + (f_{n}(t,x)-\nabla \cdot b_n)w_n +g_{n}(t,x)v_{n}+h_{n}\cdot \nabla v_{n}, \quad x\in\R^N, t>0.
$$
Thus, the variation of constant formula yields that
\begin{align}\label{aux-eq5-4}
w_n(t,\cdot)&=e^{t(\Delta-I)}w_n(0)+\underbrace{\int_{0}^{t}e^{(t-s)(\Delta-I)}\nabla\cdot (w_{n}(s,\cdot) b_{n}(s,\cdot))ds}_{I_1}\nonumber\cr
&+\underbrace{\int_{0}^{t}e^{(t-s)(\Delta-I)}((1+f_n(s,\cdot)-\nabla \cdot b_n(s,\cdot))w_n(s,\cdot)+g_{n}(s,\cdot)v_{n}(s,\cdot)+h_{n}\cdot \nabla v_{n})ds}_{I_2},
\end{align}
where $\{e^{t(\Delta-I)}\}_{t\geq 0}$ denotes the $C_0-$semigroup on $L^{p}(\R^N)$ generated by $\Delta-I$.

Observe that  $\|b_{n}(t,\cdot)\|_{\infty}\leq \mu\|u(t+t_{0n},x+x_n;t_{0n}, u_{0n}(\cdot+x_{n}))\|_{\infty}\leq \mu\frac{a_{\sup}}{b_{\inf}-\chi\mu}$. Hence, by \cite[Lemma 3.1]{SaSh1}, we have
\begin{align*}
\|I_1\|_{L^{p}(\R^N)}&\leq C\int_{0}^{t}(t-s)^{-\frac{1}{2}}e^{-(t-s)}\|\nabla b_{n}(s,\cdot)\|_{\infty}\|w_{n}(s,\cdot)\|_{L^{p}(\R^N)}ds\cr
& \leq\frac{C\mu a_{\sup}}{b_{\inf}-\chi\mu}\int_{0}^{t}(t-s)^{-\frac{1}{2}}e^{-(t-s)}\|w_{n}(s,\cdot)\|_{L^{p}(\R^N)}ds.
\end{align*}
We also observe that $\sup_{0\leq t\leq T, n\geq 1}\|1+f_{n}(t,\cdot)-\nabla\cdot b_{n}(t,\cdot)\|_{\infty} <\infty$, $
\sup_{0\leq t\leq T, n\geq 1}\|g_{n}(t,\cdot)\|_{\infty}<\infty$, and $ \sup_{0\leq t\leq T, n\geq 1}\|h_{n}(t,\cdot)\|_{\infty}<\infty$, thus we have
\begin{align*}
\|I_{2}\|_{L^p(\R^N)}& \leq C\int_{0}^{t}e^{-(t-s)}\left\{\|w_{n}(s,\cdot)\|_{L^p(\R^N)}+\|v_{n}(s,\cdot)\|_{W^{1,p}(\R^N)}\right\}ds.
\end{align*}
Since $(\Delta-\lambda I)v_{n}=-\mu w_{n}$, then by elliptic regularity, we have that
\begin{equation*}
\|v_{n}(t,\cdot)\|_{W^{2,p}(\R^N)}\leq C\|w_{n}(t,\cdot)\|_{L^{p}(\R^N)}.
\end{equation*}
Hence, since $\|e^{t(\Delta-I)}w_{n}(0,\cdot)\|_{L^p(\R^N)}\leq e^{-t}\|w_{n}(0,\cdot)\|_{L^p(\R^N)}$, we obtain
$$
\|w_{n}(t,\cdot)\|_{L^{p}(\R^N)}\leq \|w_{n}(0)\|_{L^{p}(\R^N)}+C\int_{0}^{t}(t-s)^{-\frac{1}{2}}\| w_{n}(s,\cdot)\|_{L^p(\R^N)}ds
$$
for some constant $C>0$. Therefore it follows from Generalized Gronwall's  inequality  (see Lemma 2.5 \cite{SaSh1}) that
$$
\|w_{n}(t,\cdot)\|_{L^{p}(\R^N)}\leq C_{T}\|w_{n}(0,\cdot)\|_{L^{p}(\R^N)}, \quad \forall \, 0\leq t\leq T, \ \forall\,  n\geq 1,
$$
where $C_{T}>0$ is a constant. Thus
\begin{equation}\label{aux-eq5-5}
\lim_{n\to\infty}\sup_{0\leq t\leq T}\|w_{n}(t,\cdot)\|_{L^{p}(\R^N)}=0.
\end{equation}

For $p>N$, by regularity and a priori estimates for elliptic operators, there is a constant $C>0$ such that
$$
\|(\Delta-\lambda I)^{-1}w\|_{C^{1,b}_{\rm unif}(\R^N)}\leq C\|w\|_{L^{p}(\R^N)}, \quad \forall w\in L^{p}(\R^N).
$$
Combining this with \eqref{aux-eq5-5} we have that
\begin{equation}\label{aux-eq5-6}
\lim_{n\to\infty}\sup_{0\leq t\leq T}\|v_{n}(t,\cdot)\|_{C^{1,b}_{\rm unif}(\R^N)}=0.
\end{equation}
It follows from  the claim in Step 3 that for every $n\geq 1,$
$$
0\leq\lambda v(t+t_{0n},x;t_{0n},x_n,\tilde u_{0n})\leq \frac{a_0}{2\chi},\, \text{}\,\, |\chi \nabla v(t+t_{0n},x;t_{0n},x_n,\tilde u_{0n})|\leq \frac{\varepsilon_0}{2},\quad \forall 0\leq t\leq T, \, x\in D_{L}.
$$
Thus \eqref{aux-eq5-6} implies that, for $n\gg 1$, $ \forall 0\leq t\leq T, \, x\in D_{L}$, there holds
$$
0\leq \chi\lambda v(t+t_{0n},x;t_{0n},x_n, u_{0n}(\cdot+x_n))\leq a_0,\,\, |\chi \nabla vt+t_{0n},x;t_{0n},x_n, u_{0n}(\cdot+x_n))|\leq \varepsilon_0.
$$
Hence, it follows from the arguments of \eqref{n-n0}  that
$$u(T+t_{0n},0;t_{0n},x_n, u_{0n}(\cdot+x_n)) >\delta_{n},
$$
which is a contradictions. Hence {\bf case 1} does not hold.
}

\medskip

\noindent {\bf Case 2.} $\liminf_{n\to\infty}|D_{0n}|>0$.

In this case, without lost of generality, we might suppose that $\inf_{n\geq 1}|D_{0n}|>0$, and there a suitable $N-$cube, $D\subset\subset D_{3L}$ with  $\inf_{n\geq 1}|D\cap D_{0n}|>0$. Let $\Psi_{n}(x,t)$ denotes the solution of
\begin{align}\label{aux-eq5-7}
\begin{cases}
u_t=\Delta u, \quad x\in D_{3L}\cr
u=0 , \quad \text{on}\, (0,T)\times \partial D_{3L}\cr
u(\cdot,0)=\frac{\delta_0}{2}\chi_{_{D\cap D_{0n}}}.
\end{cases}
\end{align}
Thus, by comparison principle for parabolic equations, we have
$$e^{t\Delta}u_{0n}(x+x_{n})\geq \Psi_{n}(x,t), \quad \forall x\in D_{3L}, 0\leq t\leq T,\,\, n\geq 1.$$
From this, it follows that
\begin{equation}\label{aux-eq5-8}
\|e^{t\Delta}u_{0n}(\cdot+x_{n}) \|^{2}_{C^{\infty}(D_{3L})}\geq \frac{1}{|D_{3L}|}\int_{D_{3L}}\Psi^2_{n}(x,t)dx,\quad \forall 0\leq t\leq T,\,\, n\geq 1.
\end{equation}
Note that for every $n\geq 1$, $\Psi_{n}(x,t)$ can be written as
$$
\Psi_{n}(x,t)=\frac{\delta_0}{2}\sum_{k=1}^{\infty}e^{-t\tilde{\lambda}_{k}}\phi_{k}(x)\left[\int_{D_{3L}}\phi_{k}(y)\chi_{_{D\cap D_{0n}}}(y)dy\right],
$$
where $\{\phi_{k}\}_{k\geq 1}$ denotes the orthonormal basis of $L^{2}(D_{3L})$ consisting of eigenfunctions with corresponding eigenvalues $\{\tilde{\lambda}_{k}\}$ of $-\Delta$ with Dirichlet boundary conditions on $D_{3L}$. Since $\tilde{\lambda}_{1}$ is principal, then we might suppose that $\phi_{1}(x)>0$ for every $x\in D_{3L}$.  Thus
\begin{align}\label{aux-eq5-9}
\|\Psi_{n}(\cdot,t)\|^2_{L^{2}(D_{3L})}&=\sum_{k=1}^{\infty}e^{-2t\tilde{\lambda}_{k}}\left[\frac{\delta_0}{2}\int_{D_{3L}}\phi_{k}(y)\chi_{_{D\cap D_{0n}}}(y)dy\right]^2\cr
&\geq e^{-2t\tilde{\lambda}_1}\left[\int_{D_{3L}}\phi_{1}(y)\chi_{_{D\cap D_{0n}}}(y)dy\right]^2\cr
&\geq e^{-2t\tilde{\lambda_1}}\left[\frac{\delta_0}{2}|D\cap D_{0n}| \min_{y\in D}\phi_{1}(y)\right]^2.
\end{align}
Since $\inf_{n\geq 1}|D_{0n}|>0$ and $\min_{y\in D}\phi_{1}(y)>0$, it follows from \eqref{aux-eq5-8} and \eqref{aux-eq5-9} that
$$
\inf_{0\leq t\leq T,n\geq 1} \|e^{t(\Delta-I)}u_{0n}(\cdot+x_{n})\|_{C(D_{3L})}>0.
$$

Thus there is $0<T_{0}\ll 1$ such that
$$
\inf_{n\geq 1} \|u(\cdot, T_0+t_{0n};t_{0n},x_{n},u_{0n}(\cdot+x_n))\|_{C^{0}(D_{3L})}>0.
$$
Hence, we might suppose that
$u(\cdot, T_0+t_{0n};t_{0n},x_{n},u_{0n}(\cdot+x_n))\to u^{*}_{0}$ locally uniformly and $\|u^{*}_{0}\|_{C(D_{3L})}>0$. Moreover, by Lemma \ref{continuity with respect to open compact topology},  we might assume that
$(u(\cdot, T+t_{0n};t_{0n},x_{n},u_{0n}(\cdot+x_n)),v(\cdot, T+t_{0n};t_{0n},x_{n},u_{0n}(\cdot+x_n)))\to (u^{*}(x,t),v^{*}(x,t))$, $a(x+x_{n},t)\to a^{*}(x,t)$, and  $b(x+x_{n},t)\to b^{*}(x,t)$, where $(u^{*},v^{*})$ satisfies
\begin{equation*}
\begin{cases}
u^{*}_{t}=\Delta u^{*}-\chi\nabla\cdot(u^{*}\nabla v^{*})+(a^{*}-b^{*}u^{*})u^{*}\cr
0=(\Delta-\lambda I)v^{*}+\mu u^{*} \cr
u^*(\cdot,0)=u^{*}_{0}.
\end{cases}
\end{equation*}
Sine $\|u^*_0\|_{\infty}>0$ and $u^*(x,t)\geq 0$, it follows from comparison principle for parabolic equations that $u^{*}(x,t)>0$ for every $x\in R^N$ and $t\in (0, T]$. In particular $u^{*}(0,T)>0$. Note by \eqref{aux-eq5-2} that we must have $u^{*}(0,T)=0$, which is a contradiction. Hence the result holds.
\end{proof}

We now present the proof of Theorem \ref{Main-thm1}.

\begin{proof}[Proof of Theorem \ref{Main-thm1}]
(i) Let $u_0\in C^b_{\rm uinf}(\R^N)$, with $u_{0\inf}>0$ be given. It follows from \eqref{u-upper-bound-eq3} that there is $T_1>0$ such that
\begin{equation*}
u(x,t+t_0;t_0,u_0)\le M^+:={ \frac {a_{\sup}}{b_{\inf}-\chi\mu}+1},\quad \forall \,\, t\ge T_1, \quad, \forall \ t_0\in\R.
\end{equation*}  Note that $T_1$ is independent of $t_0$.
We claim that
\begin{equation}\label{persistence-eq}
m(u_0):=\inf_{t_0\in\mathbb{R},(x,t)\in\R^N\times[0 , \infty)}u(x,t+t_0;t_0,u_0)>0.
\end{equation}
 In fact,  since $u_{0\inf}>0$, by Lemma \ref{main-lem1}, we have that
\begin{equation}\label{persistence-eq0}
\delta_1:=\inf_{t_0\in\mathbb{R},(x,t)\in\R^N\times[t_0, t_0+T_1]}u(x,t+t_0;t_0,u_0)\ge u_{0\inf}e^{-T_1(a_{\inf}+b_{\sup}\|u_0\|_{\infty}e^{T_1a_{\sup}})}>0.
\end{equation}
 Let
 $$\delta_2=\min\{\delta_1,\delta_0(T_1)\},
  $$
  where $\delta_0(T_1)$ is given by Lemma \ref{main-new-lm2}.  Then $\delta_2>0$. By induction, it follows from Lemma \ref{main-new-lm2} that
\begin{equation}\label{persistence-eq1}
\delta_2\leq \inf_{x}u(x,t_0+nT_1;t_0,u_0))\leq M^+ ,  \quad \forall\, t_0\in\mathbb{R}\,\,\, {\rm and}\,\,\, \forall\ n\in\mathbb{N}.
\end{equation}
Lemma \ref{main-lem1} implies that for every $t_0\in\mathbb{R}$,  $x\in\R^N, \ t\in[0,T_1]$ and $n\in\mathbb{N}\cup\{0\}$, we have
\begin{align}\label{persistence-eq2}
u(x,t_0+nT_1+t;t_0,u_0)&=
u(x,t_0+nT_1+t;t_0+nT_1,u(x,t_0+nT_1,t_0,u_0))\cr
&\geq \delta_2 e^{t(a_{\inf} -b_{\sup} M^+e^{T_1a_{\sup}})}\cr
&\geq \delta_2 e^{-T_1(a_{\inf} +b_{\sup} M^+e^{T_1a_{\sup}})}
\end{align}
By \eqref{persistence-eq2}, we obtain that
$$
\inf_{t_0\in\mathbb{R},(x,t)\in\R^N\times[0,\infty)}u(x,t_0+t;t_0,u_0))\geq \delta_2 e^{-T_1(a_{\inf} +b_{\sup} M^+T_1)}
$$
The last inequality  yields that $m(u_0)>0$. Hence \eqref{persistence-eq} holds.

\smallskip

  (ii) Let $(\underline{M}_{n},\overline{M}_{n})_{n\geq 0}$ be the sequence define by \eqref{M-n def}.  Let $u_0\in C^b_{\rm unif}(\R^N)$ with $u_{0\inf}>0$ be fixed.

 We first claim that for every $n\geq 0$, and ${\varepsilon}>0$ there is $T_{\varepsilon}^n(u_0)$ such that
 \begin{equation}\label{p-001}
 \underline{M}_{n}-\varepsilon\leq u(x,t+t_0;t_0,u_0)\leq \overline{M}_{n}+\varepsilon\quad \forall\,x\in\R^N,\, \forall\, t\geq T^{n}_{\varepsilon}(u_0), \, \forall\, t_0\in\R,
 \end{equation}
 which implies that for any ${\varepsilon}>0$ there is $T_{\varepsilon}(u_0)$ such that \eqref{attracting-rect-eq1} holds.

 In fact, for $n=0$, it clear that $\underline{M}_{0}=0\leq u(x,t+t_0;t_0,u_0)$ for every $x\in\R^N, \ t\geq 0$, and $t_{0}\in\R$. It follows \eqref{global-eq02} that there is $T^0_{\varepsilon}(u_0)$ such that
 $$
u(x,t+t_0;t_0,u_0)\leq \overline{M}_{1}+{\varepsilon},\quad \forall\, x\in\R^N,\, t\geq T^0_{{\varepsilon}}(u_0), \,\forall\, t_0\in\R.
 $$
 Hence \eqref{p-001} holds for $n=0$. Suppose that \eqref{p-001} holds for  $n-1$, ($n\geq 1$). We show that \eqref{p-001} also holds for $n$. Indeed, let $\varepsilon>0$. It follows from the induction hypothesis that there is $\tilde{T}^{n-1}_{\varepsilon}(u_0)\gg 1$ that
 \begin{equation}\label{p-002}
\underline{M}_{n-1}-{\frac{\varepsilon}{4}}\leq u(x,t+t_0;t_0,u_0)\leq \overline{M}_{n-1}+{\frac{\varepsilon}{4}} \quad \forall\ x\in\R^N,\,\forall\, t\geq T^{n-1}_{\varepsilon}(u_0), \, \forall\, t_0\in\R.
 \end{equation}
 This implies that
 \begin{equation}\label{eq-xx00}
\begin{cases}
u_t(\cdot,\cdot+t_0;t_0,u_0)\\
\geq  \Delta u(\cdot,\cdot+t_0;t_0,u_0) -\chi(\nabla v\cdot \nabla u)(\cdot,\cdot+t_0;t_0,u_0)\\
\quad+(a_{\inf}-\chi\mu (\overline{M}_{n-1}+\frac{\varepsilon}{4})-(b_{\sup}-\chi\mu)u(\cdot,\cdot+t_0;t_0,u_0))u(\cdot,\cdot+t_0;t_0,u_0),\,t>\tilde T^{n-1}_{\varepsilon}(u_0),\\
u(\cdot,\tilde T^{n-1}_{\varepsilon}(u_0)+t_0;t_0,u_0)\geq m(u_0),
\end{cases}
 \end{equation}
 where $m(u_0):=\inf\{u(x,t+t_0;t_0,u_0)\ |\ x\in\R^N,\,\, t\in[0 ,\infty),\,\, t_0\in\R\}>0$.
 Hence, it follows from comparison principle for parabolic equations that there is ${\tilde T}^{n}_{\varepsilon}(u_0)\ge \tilde T^{n-1}_{\varepsilon}(u_0) $ such that
 \begin{equation}\label{p-003}
 u(x,t+t_0;t_0,u_0)\geq \frac{(a_{\inf}-\chi\mu(\overline{M}_{n-1}+\frac{\varepsilon}{4}))_+}{b_{\sup}-\chi\mu}-\frac{\varepsilon}{4}, \quad \forall\ t\geq \tilde{T}^{n}_{\varepsilon}(u_0), \, x\in\R^N,\, \forall\, t_0\in\R.
 \end{equation}
 Note that, since {\bf (H2)} holds, then it follows from Lemma \ref{Main-lem2'} that
 \begin{equation}\label{p-004}
 \frac{(a_{\inf}-\chi\mu(\overline{M}_{n-1}+\frac{\varepsilon}{4}))_+}{b_{\sup}-\chi\mu}\geq \underline{M}_{n} -\frac{\chi\mu\varepsilon}{4(b_{\sup}-\chi\mu)}\geq \underline{M}_{n}-\frac{\varepsilon}{4}.
 \end{equation}
 It follows from \eqref{p-003} and \eqref{p-004} that
 \begin{equation}\label{eq-xx01}
\begin{cases}
u_t(\cdot,\cdot+t_0;t_0,u_0)\\
\leq  \Delta u(\cdot,\cdot+t_0;t_0,u_0) -\chi(\nabla v\cdot \nabla u)(\cdot,\cdot+t_0;t_0,u_0)\\
\quad+(a_{\sup}-\chi\mu (\underline{M}_{n}-\frac{\varepsilon}{2})-(b_{\inf}-\chi\mu)u(\cdot,\cdot+t_0;t_0,u_0))u(\cdot,\cdot+t_0;t_0,u_0),\,t>\tilde{T}^{n}_{\varepsilon}(u_0),\\
u(\cdot,\tilde{T}^{n}_{\varepsilon}(u_0)+t_0;t_0,u_0)\leq \overline{M}_{n-1}+\frac{\varepsilon}{4}.
\end{cases}
 \end{equation}
  Hence, it follows from comparison principle for parabolic equations that there is $T^{n}_{\varepsilon}(u_0)\ge \tilde{T}^{n}_{\varepsilon}(u_0) $ such that
 \begin{equation}\label{p-005}
 u(x,t+t_0;t_0,u_0)\leq \frac{(a_{\inf}-\chi\mu(\underline{M}_{n}-\frac{\varepsilon}{2}))_+}{b_{\sup}-\chi\mu}+\frac{\varepsilon}{2}, \quad \forall\ t\geq T^{n}_{\varepsilon}(u_0), \, x\in\R^N,\, \forall\, t_0\in\R.
 \end{equation}
 Observe that, sine {\bf (H2)} holds, Lemma \ref{Main-lem2'} implies that
 \begin{equation}\label{p-006}
 \frac{(a_{\inf}-\chi\mu(\underline{M}_{n}-\frac{\varepsilon}{2}))_+}{b_{\inf}-\chi\mu}\leq \overline{M}_{n}+\frac{\varepsilon}{2}.
 \end{equation}
 Hence, it follows from \eqref{p-003}-\eqref{p-006} that \eqref{p-001} also holds for $n$. Thus we conclude that \eqref{p-001} holds for every $n\geq 0$.

  Next, we show that the set $\mathbb{I}_{inv}$ given by \eqref{Invariant set} is an invariant set for positive solutions of \eqref{P}.

  By Lemma \ref{Main-lem2'} we have that $\underline{M}_{n}\nearrow \underline{M}$ and $\overline{M}_{n}\searrow\overline{M}$. It suffices to show that the set $\mathbb{I}^n_{inv}:=\{ u_0\in C^b_{\rm unif}(\R^N)\ |\ \underline{M}_{n}\leq u_0(x)\leq\overline{M}_{n} \}$, $n\ge 0$, is positively invariant for \eqref{P}. This is also done by induction on $n\ge 0$. The case $n=0$ is guaranteed by Theorem \ref{global-existence-tm} (i). Suppose that $\mathbb{I}^{n}_{inv}$ is a positive invariant set for \eqref{P}. Let $u_0\in\mathbb{I}^{n+1}_{inv}$. Since, by Lemma \ref{Main-lem2'}, $\overline{M}_n>\overline{M}_{n+1}$, it follows from \eqref{eq-xx00} and comparison principle for parabolic equations that
  $$
u(x,t+t_0;t_0,u_0)\geq \min\Big{\{}\underline{M}_{n+1}\ ,\ \underbrace{\frac{a_{\inf}-\chi\mu\overline{M}_{n}}{b_{\sup}-\chi\mu}}_{=\underline{M}_{n+1}}\Big{\}}=\underline{M}_{n+1}, \,\,\forall\, x\in\R^N,\,\,\forall\, t\geq0,\forall\,t_0\in\R.
  $$
  Using this last inequality, by \eqref{eq-xx01}, it follows from comparison principle for parabolic equations that $$ u(x,t+t_0;t_0,u_0)\leq \max\Big{\{}\overline{M}_{n+1}\ ,\ \underbrace{\frac{a_{\sup}-\chi\mu\underline{M}_{n+1}}{b_{\inf}-\chi\mu}}_{=\overline{M}_{n+1}}\Big{\}}=\overline{M}_{n+1}, \,\,\forall\, x\in\R^N,\,\,\forall\, t\geq0,\forall\,t_0\in\R. $$
  Thus, $\mathbb{I}^{n+1}_{inv}$ is also  a positive invariant set for \eqref{P}. The result thus follows.
\end{proof}

\section{Asymptotic spreading}

In this section, we  study the spreading properties of positive solutions and prove Theorems \ref{spreading-properties} and  \ref{spreading-properties-1}.
We first present two lemmas.

\begin{lem}
\label{spreading-lm1}
Consider
\begin{equation}
\label{spreading-eq1}
u_t=u_{xx}+{ q_0}u_x+u(a_0-b_0u),\quad x\in\R^1,
\end{equation}
where ${ q_0}\in\R^1$ is a constant  and $a_0,b_0$ are two positive constants.
Let $u(x,t;u_0)$ be the solution of \eqref{spreading-eq1} with $u(\cdot,0;u_0)=u_0(\cdot)\in C_{\rm unif}^b(\R^1)$
($u_0(x)\ge 0$). For every nonnegative front-like initial function $u_0\in C_{\rm unif}^b(\R^1)$
(i.e. $\liminf_{x\to -\infty}u_0(x)>0$ and $u_0(x)=0$ for $x\gg 1$), there hold
$$
\limsup_{t\to\infty, x\ge ct} u(x,t;u_0)=0\quad \forall\,\, c>c_0^*
$$
and
$$
\liminf_{t\to\infty,x\le ct}u(x,t;u_0)>0\quad \forall\, \, c<c_0^*,
$$
where $c_0^*=2\sqrt{a_0}-{ q_0}$.
\end{lem}

\begin{proof}
Let $v(x,t;u_0)=u(x-{q_0}t,t;u_0)$. Then $v(x,t)$ satisfies
\begin{equation}
\label{spreading-eq2}
v_t=v_{xx}+v(a_0-b_0v),\quad x\in\R^1.
\end{equation}
By  \cite[Theorems 16 and 17]{KPP} and comparison principle for parabolic equations, for  every nonnegative front-like initial function
$u_0\in C_{\rm unif}^b(\R^1)$,
$$
\limsup_{t\to\infty,x\ge ct}v(x,t;u_0)=0\quad \forall\,\, c>2\sqrt{a_0}
$$
and
$$
\liminf_{t\to\infty,x\le ct} v(x,t;u_0)>0\quad \forall\, \, c<2\sqrt {a_0}.
$$
Note that $u(x,t;u_0)=v(x+q_0t,t;u_0)$.
The lemma thus follows.
\end{proof}

\begin{lem}
\label{spreading-lm2}
Consider
\begin{equation}
\label{spreading-eq3}
u_t=\Delta u+q_0(x,t)\cdot \nabla u+u(a_0-b_0 u),\quad x\in\R^N,
\end{equation}
where $q_0\in\R^N$ is a continuous vector function and $a_0,b_0$ are positive constants.
Let $u(x,t;u_0)$ be the solution of \eqref{spreading-eq3} with $u(\cdot,0;u_0)=u_0(\cdot)\in C_{\rm unif}^b(\R^N)$
($u_0(x)\ge 0$). If
\begin{equation}
\label{spreading-eq4}
\liminf_{|x|\to\infty}\inf_{t\ge 0}(4a_0-|q_0(x,t)|^2)>0,
\end{equation}
then for any nonnegative initial function $u_0\in C_{\rm unif}^b(\R^N)$ with nonempty support,
$$
\liminf_{t\to\infty,|x|\le ct}u(x,t;u_0)>0\quad \forall\,\, 0<c<c_0^*,
$$
where $c_0^*=\liminf_{|x|\to\infty}\inf_{t\ge 0} (2\sqrt{a_0}-|q_0(x,t)|)$.
\end{lem}

\begin{proof}
It follows from Theorem 1.5  in  \cite{Berestycki1}.
\end{proof}

Next, we prove Theorem  \ref{spreading-properties}.

\begin{proof}[Proof of Theorem \ref{spreading-properties}]
(1) Let $t_{0}\in\R$ and $u_0\in C^b_{\rm unif}(\R^N)$ with $u_0\ge 0$ such that there is $R\gg 1$ with $u_0(x)=0$ for all $\|x\|\geq R.$  By  \eqref{u-upper-bound-eq3},  for every $\varepsilon>0$ there is $T_{\varepsilon}>0$ such that
\begin{equation}\label{bb-eq01}
\|u(\cdot,t+t_0;t_0,u_0)\|\leq \frac{a_{\sup}}{b_{\inf}-\chi\mu}+\varepsilon,\quad \forall t\geq T_{\varepsilon}.
\end{equation}
By \eqref{v-bound-eq2} and  \eqref{bb-eq01}, we have that
\begin{equation}\label{bb-eq02}
\|\nabla v(\cdot,t+t_0;t_0,u_0)\|_{\infty}\leq \frac{\mu\sqrt{N}}{2\sqrt{\lambda}}\big(\frac{a_{\sup}}{b_{\inf}-\chi\mu}+\varepsilon\big),\quad \forall t\geq T_{\varepsilon}.
\end{equation}
Choose $C>0$ such that
$$
u_0(x)\leq Ce^{-\sqrt{a_{\sup}}|x|}, \quad \forall\ x\in\R^N,
$$
and let $$
K_{\varepsilon}:=\sup_{0\leq t\leq T_{\varepsilon}}\|\nabla v(\cdot,t+t_0;t_0,u_0)\|_{\infty}.
$$
Let $\xi\in \mathbb{S}^{N-1}$ be given and consider
$$
\overline{U}(x,t;\xi):=Ce^{-\sqrt{a_{\sup}}(x\cdot\xi-(2\sqrt{a_{\sup}}+\chi K_{\varepsilon})t)}.
$$
Recall from inequality \eqref{global-eq01} that $
u_{t}(\cdot,\cdot+t_0;t_0,u_0)\leq \mathcal{L}_{0}u(\cdot,\cdot+t_0;t_0,u_0),
$ where
$$\mathcal{L}_{s}(w):= \Delta w-\chi\nabla v(\cdot,\cdot+s+t_0;t_0,u_0)\cdot\nabla w +(a_{\sup}-(b_{\inf}-\chi\mu)w)w, \forall \ w\in C^{2,1}(\R^N\times(0,\infty)), \forall\ s\geq 0.$$
We have that
\begin{align*}
\overline{U}_t-\mathcal{L}_{0}\overline{U}&=\left((2a_{\sup}+\chi\sqrt{a_{\sup}}K_{\varepsilon})-a_{\sup}+\chi\sqrt{a_{\sup}}\xi\cdot\nabla v(\cdot,\cdot+t_0;t_0,u_0)-(a_{\sup}-(b_{\inf}-\chi\mu)\overline{U})\right)\overline{U}\nonumber\\
&=\left(\chi\sqrt{a_{\sup}}(K_{\varepsilon}-\xi\cdot\nabla v(\cdot,\cdot+t_0;t_0,u_0)) +(b_{\inf}-\chi\mu)\overline{U}\right)\overline{U}\nonumber\\
&\geq \left(\chi\sqrt{a_{\sup}}(K_{\varepsilon}-\|\xi\cdot\nabla v(\cdot,\cdot+t_0;t_0,u_0)\|_{\infty}) +(b_{\inf}-\chi\mu)\overline{U}\right)\overline{U}\nonumber\\
&\geq 0.
\end{align*}
Since $u_0(x)\leq \overline{U}(x,0;\xi)$, then it follows from comparison principle for parabolic equations that
\begin{equation}\label{bb-eq03}
u(x,t+t_0;t_0,u_0)\leq \overline{U}(x,t;\xi),\quad \forall \ x\in\R^N, \ \forall \ t\in[0 , T_{\varepsilon}
], \ \forall\ \xi\in \mathbb{S}^{N-1}.
\end{equation}
Next, let
$$ L_{\varepsilon}:=\frac{\mu\sqrt{N}}{2\sqrt{\lambda}}\big(\frac{a_{\sup}}{b_{\infty}-\chi\mu}+\varepsilon\big)$$
and
$$
\overline{W}(x,t;\xi):=e^{-\sqrt{a_{\sup}}(\xi\cdot x-(2\sqrt{a_{\sup}}+\chi L_{\varepsilon})t)}\overline{U}(0,T_{\varepsilon};\xi), \quad \forall t\geq 0, \ \forall\ x\in\R^N, \ \forall\ \xi\in\mathbb{S}^{N-1}.
$$
Similarly, using inequality \eqref{bb-eq02}, we have that
\begin{align*}
\overline{W}_t-\mathcal{L}_{T_{\varepsilon}}\overline{W}&=\left(\chi\sqrt{a_{\sup}}(L_{\varepsilon}-\xi\cdot\nabla v(\cdot,\cdot+T_{\varepsilon}+t_0;t_0,u_0) +(b_{\inf}-\chi\mu)\overline{W}\right)\overline{W}\nonumber\\
&\geq \left(\chi\sqrt{a_{\sup}}(L_{\varepsilon}-\|\xi\cdot\nabla v(\cdot,\cdot+T_{\varepsilon}+t_0;t_0,u_0)\|_{\infty}) +(b_{\inf}-\chi\mu)\overline{W}\right)\overline{W}\nonumber\\
&\geq 0.
\end{align*}
But by \eqref{bb-eq03}, we have that $\overline{W}(x,0;\xi)= \overline{U}(x,T_{\varepsilon};\xi)\geq u(\cdot, T_{\varepsilon}+t_0;t_0,u_0)$. Hence by comparison principle for parabolic equations we obtain that
\begin{equation}\label{bb-eq04}
u(x,t+t_0;t_0,u_0)\leq \overline{W}(x,t;\xi),\quad \forall \ x\in\R^N, \ \forall \ t\geq T_{\varepsilon}
, \ \forall\ \xi\in \mathbb{S}^{N-1}.
\end{equation}
Observe that
$$
\lim_{\varepsilon\to 0^+}(2\sqrt{a_{\sup}}+\chi L_{\varepsilon})=2\sqrt{a_{\sup}}+ \frac{\chi\mu\sqrt{N}a_{\sup}}{2(b_{\inf}-\chi\mu)\sqrt{\lambda}}=c^{*}_{+}(a,b,\chi,\lambda,\mu).
$$
Thus, it follows from \eqref{bb-eq04} and the  definition of $\overline{W}$ that
$$
\lim_{t\to\infty}\sup_{|x|\geq ct}u(x,t+t_0;t_0,u_0)=0,
$$
whenever $c>c^*_+(a,b,\chi,\lambda,\mu)$. This complete the proof of (1).

\medskip

(2) We first claim that
\begin{equation}\label{cc-eq00}
 4(a_{\inf}-\frac{\chi\mu a_{\sup}}{b_{\inf}-\chi\mu} )-\frac{N\mu^2\chi^{2} a^2_{\sup}}{4\lambda(b_{\inf}-\chi\mu)^{2}}> 0.
\end{equation}
Indeed, let $\tilde \mu=\frac{\chi\mu a_{\sup}}{b_{\inf}-\chi\mu}$. \eqref{cc-eq00} is equivalent to
$4(a_{\inf}-\tilde \mu)-\frac{N}{4\lambda} \tilde \mu^2>0.$
This implies that
$$
0<\tilde \mu=\frac{\chi\mu a_{\sup}}{b_{\inf}-\chi\mu}<\frac{2 a_{\inf}}{1+\sqrt{1+\frac{Na_{\inf}}{4\lambda}}}
$$
and then
$$
\frac{b_{\inf}-\chi\mu}{\chi\mu}>\frac{\Big(1+\sqrt{1+\frac{Na_{\inf}}{4\lambda}}\Big)a_{\sup}}{2a_{\inf}}.
$$ This proves the claim.

Next, by \eqref{u-upper-bound-eq3}, \eqref{v-bound-eq1}, and  \eqref{v-bound-eq2}, for every $\varepsilon>0$, we can choose $T_{\varepsilon}$ with $T_\varepsilon\to \infty$ as $\varepsilon\to 0$ such that
\begin{equation}
\label{aux-eq01}
\|v(\cdot,t+t_0;t_0,u_0)\|_{\infty}<\frac{\mu a_{\sup}}{\lambda(b_{\inf}-\chi\mu)}+\varepsilon \,\, \text{and}\,\, \|\nabla v(\cdot,t+t_0;t_0,u_0)\|_{\infty}<\frac{\mu\sqrt{N}}{2\sqrt{\lambda}}\Big(\frac{a_{\sup}}{b_{\inf}-\chi\mu}+\varepsilon \Big)
\end{equation}
for all $t\ge T_{\varepsilon}$.

Note that for every $\varepsilon>0$ and $t\geq T_{\varepsilon}+t_0$, we have
\begin{align}\label{cc-eq02}
u_{t}(\cdot,\cdot;t_0,u_0)&\geq \Delta u(\cdot,\cdot;t_0,u_0) -\chi\nabla v(\cdot,\cdot;t_0,u_0)\cdot\nabla u(\cdot,\cdot;t_0,u_0)\nonumber\\
&+(a_{\inf}-\frac{\chi\mu a_{\sup}}{b_{\inf}-\chi\mu}-\chi\mu\varepsilon-(b_{\sup}-\chi\mu)u(\cdot,\cdot;t_0,u_0))u(\cdot,\cdot;t_0,u_0).
\end{align}
For every $\varepsilon>0$,  let $U(\cdot,\cdot;\varepsilon)$ denotes the solution of the initial value problem
\begin{equation}\label{cc-eq03}
\begin{cases}
U_{t}(\cdot,\cdot;\varepsilon)=\mathcal{A}_{\varepsilon}(U)(\cdot,\cdot;\varepsilon), \ t>0, \ x\in\R^N\cr
U(\cdot,0;\varepsilon)=u(\cdot,T_{\varepsilon}+t_0;t_0),
\end{cases}
\end{equation}
where
$$
\mathcal{A}_{\varepsilon}(U)(\cdot,\cdot;\varepsilon)=\Delta U(\cdot,\cdot;\varepsilon)+q(\cdot,\cdot;\varepsilon)\cdot\nabla U(\cdot,\cdot;\varepsilon)+ U(\cdot,\cdot;\varepsilon)F_{\varepsilon}(U(\cdot,\cdot;\varepsilon)),
$$
$$
F_{\varepsilon}(s)=a_{\inf}-\frac{\chi\mu a_{\sup}}{b_{\inf}-\chi\mu}-\chi\mu\varepsilon -(b_{\sup}-\chi\mu)s, \quad \forall s\in \R.
$$
and
$$
q(x,t;\varepsilon)=\begin{cases}
-\chi \nabla v(\cdot,t+T_{\varepsilon}+t_0;t_0,u_0), \  t\geq 0\\
-\chi \nabla v(\cdot,T_{\varepsilon}+t_0;t_0,u_0),\ t<0.
\end{cases}
$$
Hence, by comparison principle for parabolic equations, it follows from \eqref{cc-eq02} and \eqref{cc-eq03} that
\begin{equation}\label{cc-eq04}
u(x,t+T_{\varepsilon}+t_0;u_0)\geq U(x,t;\varepsilon), \quad \varepsilon>0, \ t\geq0, \ x\in\R^N.
\end{equation}
Observe that for $0<\varepsilon\ll 1$, since {\bf (H3)} holds,  it follows form \eqref{cc-eq00} and  \eqref{aux-eq01} that
\begin{equation}
\lim_{R\to\infty}\inf_{t\geq 0,|x|\geq R}(4F_{\varepsilon}(0)-\|q(x,t;\varepsilon)\|^2)\geq 4F_{\varepsilon}(0)-\chi^2\frac{\chi^2\mu^2 N}{4\lambda}\Big(\frac{a_{\sup}}{b_{\inf}-\chi\mu}+\varepsilon \Big)^2>0.
\end{equation}

By Lemma \ref{spreading-lm2}, it holds that
\begin{equation}\label{cc-eq05}
\liminf_{t\rightarrow\infty}\inf_{|x|\leq ct}U(x,t;\varepsilon)>0
\end{equation}
for every $0<\varepsilon\ll 1$ and  $0\leq c < c^{\ast}_{\varepsilon}$ where
\[
c^{\ast}_{\varepsilon}:=\liminf_{|x|\rightarrow\infty}\inf_{t\geq T_{\varepsilon}}\left(2\sqrt{a_{\inf}-\frac{\chi\mu a_{\sup}}{b_{\inf}-\chi\mu}-\chi\mu\varepsilon}-\chi\|\nabla v(x,t+t_0;t_0,u_0)\|\right).
\]
Combining inequalities \eqref{cc-eq04} and \eqref{cc-eq05}, we obtain that
\begin{equation}\label{cc-eq06}
\liminf_{t\rightarrow\infty}\inf_{|x|\leq ct}u(x,t+T_{\varepsilon}+t_0;t_0,u_0)>0\quad \forall\,\, 0<\varepsilon\ll 1,\,\, \forall\, 0\leq c < c^{\ast}_{\varepsilon}.
\end{equation}
 Using \eqref{aux-eq01}, we have that
 \begin{equation*}
c^*_{\varepsilon}\geq 2\sqrt{a_{\inf}-\frac{\chi\mu a_{\sup}}{b_{\inf}-\chi\mu}-\chi\mu\varepsilon}-\chi\frac{\mu\sqrt{N}}{2\sqrt{\lambda}}\Big(\frac{a_{\sup}}{b_{\inf}-\chi\mu}+\varepsilon \Big)
 \end{equation*}
Hence
\begin{equation}\label{cc-eq07}
\liminf_{\varepsilon\to 0^+}c^*_{\varepsilon}\geq 2\sqrt{a_{\inf}-\frac{\chi\mu a_{\sup}}{b_{\inf}-\chi\mu}}-\frac{\chi\mu\sqrt{N}a_{\sup}}{2\sqrt{\lambda}(b_{\inf}-\chi\mu)} :=c_{-}^*(a,b,\chi,\lambda,\mu)
\end{equation}
This together with  \eqref{cc-eq06} implies  \eqref{lower-bound-spreading-speed}.
\end{proof}

We now prove  Theorem  \ref{spreading-properties-1}.

\begin{proof}[Proof of Theorem \ref{spreading-properties-1}]
(1) For given $\xi\in S^{N-1}$ and nonnegative front-like initial function
$u_0\in C_{\rm unif}^b(\R^N)$ in the direction of $\xi$,  there is $C>0$ such that
$$
u_0(x)\leq Ce^{-\sqrt{a_{\sup}}x\cdot\xi}, \quad \forall\ x\in\R^N.
$$
For any $\varepsilon>0$, let  $T_{\varepsilon}>0$ be as in \eqref{bb-eq01}.
As in the proof of Theorem \ref{spreading-properties}(1), let $$
K_{\varepsilon}:=\sup_{0\leq t\leq T_{\varepsilon}}\|\nabla v(\cdot,t+t_0;t_0,u_0)\|_{\infty}
$$
and
$$
\overline{U}(x,t;\xi):=Ce^{-\sqrt{a_{\sup}}(x\cdot\xi-(2\sqrt{a_{\sup}}+\chi K_{\varepsilon})t)}.
$$
By the arguments of  Theorem \ref{spreading-properties}(1),
\begin{equation*}
u(x,t+t_0;t_0,u_0)\leq \overline{U}(x,t;\xi),\quad \forall \ x\in\R^N, \ \forall \ t\in[0 , T_{\varepsilon}
].
\end{equation*}
As in the proof of Theorem \ref{spreading-properties}(1) again,
 let
$$ L_{\varepsilon}:=\frac{\mu\sqrt{N}}{2\sqrt{\lambda}}\big(\frac{a_{\sup}}{b_{\infty}-\chi\mu}+\varepsilon\big)$$
and
$$
\overline{W}(x,t;\xi):=e^{-\sqrt{a_{\sup}}(\xi\cdot x-(2\sqrt{a_{\sup}}+\chi L_{\varepsilon})t)}\overline{U}(0,T_{\varepsilon};\xi), \quad \forall t\geq 0, \ \forall\ x\in\R^N.
$$
By the arguments of  Theorem \ref{spreading-properties}(1) again,
\begin{equation*}
u(x,t+t_0;t_0,u_0)\leq \overline{W}(x,t;\xi),\quad \forall \ x\in\R^N, \ \forall \ t\geq T_{\varepsilon}.
\end{equation*}
(1) then follows.

(2) For given $\xi\in S^{N-1}$ and nonnegative front-like initial function
$u_0\in C_{\rm unif}^b(\R^N)$ in the direction of $\xi$, for any $\varepsilon> 0$,
let  $T_{\varepsilon}>0$ be such that \eqref{bb-eq01}  and \eqref{aux-eq01} hold.
By the arguments of Theorem \ref{spreading-properties}(2), for $t\geq T_{\varepsilon}+t_0$, we have
\begin{align*}\label{cc-eq02}
u_{t}(\cdot,\cdot;t_0,u_0)&\geq \Delta u(\cdot,\cdot;t_0,u_0) -\chi\nabla v(\cdot,\cdot;t_0,u_0)\cdot\nabla u(\cdot,\cdot;t_0,u_0)\nonumber\\
&+(a_{\inf}-\frac{\chi\mu a_{\sup}}{b_{\inf}-\chi\mu}-\chi\mu\varepsilon-(b_{\sup}-\chi\mu)u(\cdot,\cdot;t_0,u_0))u(\cdot,\cdot;t_0,u_0).
\end{align*}

{We claim that $\liminf_{x\cdot\xi\to -\infty} u(x,T_{\varepsilon}+t_0;t_0,u_0)>0$. Indeed, suppose that this claim is false. Then there is a sequence $\{x_{n}\}_{n\geq 1}$, such that $x_{n}\cdot\xi\to -\infty$ and $u(x_n,T_{\varepsilon}+t_0;t_0,u_0)\to 0$  as $n\to \infty$. For each $n\geq 1$, let $u_n(x,t)=u(x+x_n,t+t_0;t_0,u_0)$ and $v_n(x,t):=v(x+x_n,t+t_0;t_0,u_0)$. Observe that $u_n(x,0)=u_{0}(x+x_n)$. So, since $u_0\in C^b_{\rm unif}(\R^N)$, then $u_{n}(\cdot,0)$ are uniformly bounded and equicontinuous. By Arzela-Ascoli Theorem, without loss of generality, we may suppose that $u_{n}\to u^{**}_0$ as $n\to \infty$ locally uniformly to some $u^{**}_0\in C^b_{\rm unif}(\R^N)$. Thus, without loss of generality, by Lemma \ref{continuity with respect to open compact topology} we may suppose that $(u_n(x,t),v_n(x,t),a(x+x_n,t+t_0),a(x+x_n,t+t_0))\to (u^{**}(x,t),v^{**}(x,t),a^{**}(x,t),b^{**}(x,t))$ as $n\to \infty $ in the open compact topology. Moreover, the function $(u^{**},v^{**})$ is the classical solution of the PDE
$$
\begin{cases}
u^{**}_t=\Delta u^{**} -\chi\nabla\cdot(u^{**}\nabla v^{**}) +(a^{**}(x,t)-b^{**}(x,t)u^{**})u^{**}, \quad x\in\R^n,\, \, t>0\cr
0=(\Delta-\lambda I)v^{**}+\mu u^{**},\quad x\in\R^n,\, \, t>0,\cr
u^{**}(x,0)=u^{**}_0(x), \quad x\in\R^N.
\end{cases}
$$
But, by the hypothesis on $u_0$, there is some $\delta_0>0$ and $R_0\gg 1$ such that
$$
u_{0}(x)\geq \delta_0, \quad \text{whenever }\ x\cdot\xi\leq -R_0.
$$
Since for every $x\in\R^N$, $(x+x_n)\cdot \xi\to -\infty$ as $n\to\infty$, then $u_{0}^{**}(x)\geq \delta_0$  for every $x\in\R^N$. By Theorem \ref{Main-thm1} (i), there is $m(u^{**}_0)>0$ such that
\begin{equation}\label{aux-eq-0-0-1}
m(u_0^{**})\leq u^{**}(x,t)\leq \max\{\|u_0\|_{\infty}, \frac{a_{\sup}}{b_{\inf}-\chi\mu}\} \quad\forall\, x\in\R^N, \,\, \forall\, t\geq 0.
\end{equation}
On the other hand, we have that $u^{**}(0,T_{\varepsilon})=\lim_{n\to\infty}u(x_n,T_{\varepsilon}+t_0;t_0,u_0)=0$. Which contradicts \eqref{aux-eq-0-0-1}. Thus the claim holds.

}

 Hence there is a bounded, continuous, and non-increasing function $\phi(s)$ with
$\lim_{s\to -\infty}\phi(s)>0$ and $\phi(s)=0$ for $s\gg 0$ such that
$$
u(x,T_{\varepsilon}+t_0;t_0,u_0)\ge \phi(x\cdot\xi)\quad \forall \,\, x\in\R^N.
$$

Let $u^*(x,t)$ be the solution of \eqref{spreading-eq1} with $q_0= \chi\frac{\mu\sqrt{N}}{2\sqrt{\lambda}}\Big(\frac{a_{\sup}}{b_{\inf}-\chi\mu}+\varepsilon \Big)$,
$a_0=a_{\inf}-\frac{\chi\mu a_{\sup}}{b_{\inf}-\chi\mu}-\chi\mu\varepsilon$,  $b_0=b_{\sup}-\chi\mu$, and $u^*(x,0)=\phi(x)$.
By comparison principle for parabolic equations, we have
$$
u_{x}^*(x,t)<0\quad \forall\,\, x\in\R,\,\,\, t>0.
$$

Let $U^*(x,t+T_{\varepsilon}+t_0)=u^*(x\cdot\xi,t)$ for $x\in\R^N$ and $t\ge 0$. Then for $t\ge T_{\varepsilon}+t_0$,
\begin{align*}
U_{t}^*(\cdot,\cdot)&\le \Delta U^*(\cdot,\cdot) -\chi\nabla v(\cdot,\cdot;t_0,u_0)\cdot\nabla U^*(\cdot,\cdot)\nonumber\\
&+(a_{\inf}-\frac{\chi\mu a_{\sup}}{b_{\inf}-\chi\mu}-\chi\mu\varepsilon-(b_{\sup}-\chi\mu)U^*(\cdot,\cdot))U^*(\cdot,\cdot).
\end{align*}
Observe  that $U^*(x,T_{\varepsilon}+t_0)\le u(x,T_{\varepsilon}+t_0;t_0,u_0)$. Then, by comparison principle for parabolic equations,
$$
u(x,t;t_0,u_0)\ge U^*(x,t)\quad \forall\,\, x\in\R^N,\,\, t\ge T_{\varepsilon}+t_0.
$$
By Lemma \ref{spreading-lm1},
$$
\liminf_{t\to\infty,x\cdot\xi\le ct} u(x,t;t_0,u_0)>0\quad \forall\,\, 0<c<c_\varepsilon^*,
$$
where
  \begin{equation*}
c^*_{\varepsilon}=2\sqrt{a_{\inf}-\frac{\chi\mu a_{\sup}}{b_{\inf}-\chi\mu}-\chi\mu\varepsilon}-\chi\frac{\mu\sqrt{N}}{2\sqrt{\lambda}}\Big(\frac{a_{\sup}}{b_{\inf}-\chi\mu}+\varepsilon \Big).
 \end{equation*}
 Letting $\varepsilon \to 0$, (2) follows.
\end{proof}

\end{document}